\documentclass[reqno, 11pt, a4paper]{amsart} 
\usepackage[english]{babel}
\usepackage{amsfonts, amsmath, amsthm, amssymb,amscd,indentfirst}
\usepackage{mathtools}
\usepackage{times}
\usepackage{enumerate}
\usepackage{enumitem}
\usepackage{esint}
\usepackage{anysize} 
\usepackage{mathrsfs}
\marginsize{2.8cm}{2.8cm}{2.5cm}{2.5cm}

\newtheorem{theorem}{Theorem}[section]
\newtheorem{proposition}[theorem]{Proposition}
\newtheorem{lemma}[theorem]{Lemma}
\newtheorem{definition}[theorem]{Definition}

\newtheorem{remark}[theorem]{Remark}

 \numberwithin{equation}{section}
 
\newcommand{\past}{p^\ast _{\alpha , \gamma}}
\newcommand{\na}{N _{\alpha , \gamma}}

 \newcommand{\dt}{\textnormal{d} t}

 \newcommand{\dvta}{\textnormal{d} \vartheta _\alpha}
 \newcommand{\dvtao}{\textnormal{d} \vartheta _{\gamma}}
 \newcommand{\dvtb}{\textnormal{d} \vartheta _\beta}
 \newcommand{\vta}{\vartheta _\alpha}
  \newcommand{\vtao}{\vartheta _{\gamma}}
  \newcommand{\vtb}{\vartheta _\beta}

\title{Renormalized Solutions for Quasilinear Elliptic Equations with Robin Boundary Conditions, Lower-Order Terms, and $L^1$ Data}

\author{Juan A. Apaza }

\address{Juan A. Apaza, Departamento de Matem\'atica, Universidade Federal de São Carlos, 13565-905, São Carlos--SP, Brazil}

\email{juanpabloalconapaza@gmail.com}

\author{Manassés de Souza}

\address{Manassés de Souza, Departamento de Matem\'atica, Universidade Federal da Para\'\i ba, 58051-900, Jo\~ao Pessoa--PB, Brazil}

\email{manasses.xavier@academico.ufpb.br}

\begin{document}
\maketitle
\begin{abstract}
In this paper, we establish the existence of a solution for a class of quasilinear equations characterized by the prototype:
\begin{equation*} 
\left\{ 
\begin{aligned}
-\operatorname{div}(\vta |\nabla u| ^{p-2} \nabla u)+ \vtao b|\nabla u|^{p-1}+ \vtao c |u|^{r-1} u&= f \vta & & \text { in } \Omega, \\ 
 \vta |\nabla u| ^{p-2} \nabla u \cdot \nu + \vtb |u|^{p-2}u &= g \vtb & &  \text { on } \partial \Omega.
\end{aligned}
\right.
\end{equation*}  
Here, $\Omega$ is an open subset of $\mathbb{R}^N$ with a Lipschitz boundary, where $N\geq 2$ and $1 < p < N$. We define $\vartheta_a(x) = (1 + |x|)^a$ for $a \in (-N, (p-1)N)$, and the constants $\alpha, \beta, \gamma, r$ satisfy suitable conditions. Additionally, $f$ and $g$ are measurable functions, while $b$ and $c$ belong to a Lorentz space. Our approach also allows us to establish stability results for renormalized solutions.
\end{abstract}

\let\thefootnote\relax
\let\thefootnote\relax

\markright{EXISTENCE OF RENORMALIZED SOLUTIONS}


\section{Introduction}

In this paper we consider a class of problems with the form
\begin{equation}\label{5}
\small \left\{ 
\begin{aligned}
-\operatorname{div}((1+|x|)^{\alpha} A(\nabla u))+ (1+|x|)^{\gamma} H(x, \nabla u)+ (1+|x|)^{\gamma} G(x, u)&= f (x)(1+|x|)^{\alpha} & & \text { in } \Omega, \\ 
 (1+|x|)^{\alpha} A(\nabla u) \cdot \nu + (1+|x|)^{\beta} K(u) &= g (x) (1+|x|)^{\beta} & &  \text { on } \partial \Omega,
\end{aligned}
\right.
\end{equation}
where $\Omega$ is an open subset of $\mathbb{R}^N$ with Lipschitz boundary, and $N\geq 2$. Furthermore, $f\in L^1 (\Omega ; (1+|x|)^{\alpha})$, $g\in L^1 (\partial \Omega ; (1+|x|)^{\beta})$, 
$A, H, G, K$, and the constants $\alpha , \beta , \gamma$  satisfy suitable conditions.

When $f$ belongs to the dual space of the Sobolev space $W^{1, p}(\Sigma)$, and $\Sigma \subset \mathbb{R}^N$ is a bounded domain, the existence and uniqueness (up to additive constants) of weak solutions to the problem
\begin{equation}\label{162}
\left\{
\begin{aligned}
 -\operatorname{div}( |\nabla u| ^{p-2}\nabla u ) &= f & & \text { in } \Sigma, \\ 
 |\nabla u|^{p-2} \frac{\partial u }{\partial \nu} &= 0  & & \text { on } \partial \Sigma,
\end{aligned}
\right.
\end{equation}
are consequences of the classical theory of pseudo-monotone operators  (cf. \cite{leray1965quelquesresulatat, lions1969quelquesmethodes}). However, if $f$ is just an $L^1$-function and not an element of the dual space of $W^{1, p}(\Sigma)$, one has to give meaning to the notion of solution. When Dirichlet boundary conditions are prescribed, various definitions of solutions to nonlinear elliptic equations with the right-hand side in $L^1$ or measure have been introduced. In \cite{benilan1995L1theory, dallaglio1996approximatesoutil1, lionsmuratmanuscriptsurlessolutions, murat1994equatioslinearesavecl1}, different notions of solutions are defined, though they prove to be equivalent, at least when the datum is an $L^1$-function. The study of existence or uniqueness for Dirichlet boundary value problems has been the subject of several papers.  We recall that the linear case has been studied in \cite{stampacchia1965problemecoefdiscont}, while the nonlinear case was initially addressed in \cite{boccardogall1989nonlinearparab, boccardo1992nonlinearhandsidemeasure} and has been further explored in various contributions, including \cite{alvinomercaldo2008nonlinearsymetriza, benali2006noncoerciveintegra, benilan1995L1theory, bettamercalmura2002measuredatum, betta2002existenrenor, dalmasoorsinda1999renorm, dallaglio1996approximatesoutil1, guibemercal2006noncoercive, guibemercal2008existenceofrenormal}.

The existence for Neumann boundary value problems \eqref{162} with $L^1$-data has been addressed in various contexts. In \cite{andreu1997quasiellipandparab, chabrowski2007neumanndata, droniou2000solvingdualitymethid, droniou2009noncoerciveneuamnnboundary, prognet1997condtauxlimhomo}, the existence of a distributional solution belonging to a suitable Sobolev space, with a null mean value, is proved. However, when $p \leq 2-1/N$, the distributional solution to problem \eqref{162} does not belong to a Sobolev space and, in general, is not a summable function. For example, consider the Dirac mass at the center of a ball as the right-hand side; see \cite[Example 2.16.]{dalmasoorsinda1999renorm}.


The concept of a renormalized solution in the context of variable exponents was studied in \cite{wittboldzimmer2010exisrenorvariablel1}, where homogeneous Dirichlet boundary conditions were considered. In \cite{azroulbarbara2013renormalizedp(x)boundrcond}, a concept of renormalized solution was proposed for a Neumann problem in bounded domains with nonnegative measures in $L^1$:
\begin{equation}\label{163}
\left\{
\begin{aligned}
-\operatorname{div}(|\nabla u|^{p(x)-2} \nabla u)  +|u|^{p(x)-2} u+ b_1(u)|\nabla u|^{p(x)}&=f & & \text { in } \Sigma, \\ 
|\nabla u|^{p(x)-2} \frac{\partial u}{\partial \nu}+ b_2 (u)& =g & &\text { on } \partial \Sigma,
\end{aligned}\right.
\end{equation}
where $p \in C(\bar{\Sigma})$, $1 < \inf_{\Sigma} p$, $N \geq 3$,  $f \in L^1(\Omega)$, $g \in L^1(\partial \Omega)$, and $b_1$ and $b_2$ are increasing continuous functions with $b_1 (0) = b_2 (0) = 0$, see also \cite{ibrangoouaro2016anisotropic}.

In \cite{nyanquiniouar2012entropyfouri}, a type of Fourier boundary problem is studied, and the existence and uniqueness of a renormalized solution to the problem are given when the data $f \in L^1(\Omega)$ and $g \in L^1(\partial \Omega)$:
$$
\left\{ \begin{aligned}
-\operatorname{div}(|\nabla u|^{p(x)-2} \nabla u) + b(u) &=f & & \text { in } \Sigma, \\ 
|\nabla u|^{p(x)-2} \frac{\partial u}{\partial \nu}+ c u &=g & & \text { on } \partial \Sigma,
\end{aligned} \right.
$$
where $p \in C(\bar{\Sigma})$, $1 < \inf_{\Sigma} p$, $N \geq 3$, $c > 0$, and $b$ is a surjective and nondecreasing function such that $b(0) = 0$. To obtain their results, they define a new space that helps to account for the boundary conditions.

For nonlinear equations with a Radon measure in the right-hand side,
\begin{equation}\label{164}
\left\{
\begin{aligned}
-\operatorname{div} ( |\nabla u|^{p-2}\nabla u ) +  b(x)|\nabla u|^\lambda &=\mu & & \text { in } \Sigma, \\
 u&=0 & & \text { on } \partial \Sigma,
 \end{aligned}
 \right.
 \end{equation}
where $N \geq 2$, $1<p<N$, $0 \leq \lambda \leq p-1$, and  $b$ belongs to the Lorentz space $L^{N, 1}(\Sigma)$. In \cite{betta2002existenrenor}, an existence result for the problem \eqref{164} is provided.


The problem \eqref{164} was  studied in \cite{boccardogall1989nonlinearparab, boccardo1992nonlinearhandsidemeasure} (and in \cite{vecchio1995nonlinearmeasuredata} where a term $b(x)|\nabla u|^{p-1}$ is considered). In these papers, the existence of a solution that satisfies the equation in the distributional sense is proven when $p>2-1 / N$.  This assumption on $p$ ensures that the solution belongs to the Sobolev space $W_0^{1, q}(\Sigma)$ with $q < N(p-1)/(N-1)$ (compare this with Remark \ref{174} below).   To address the problem \eqref{164}, two equivalent notions of solutions have been introduced: the notion of entropy solution in \cite{benilan1995L1theory, boccardogall1996existentrpy}, and the notion of renormalized solution in \cite{lionstoappearrenormaliseesnonlineares, murat1993renormalizedednolineal, murat1994equatioslinearesavecl1}. In the case where the measure $\mu$ belongs to $L^1(\Sigma)$ or to $L^1(\Sigma)+W^{-1, p^{\prime}}(\Sigma)$, these papers prove the existence and uniqueness of such solutions.

By employing the arguments in \cite{betta2002existenrenor}, we obtain estimates of solutions for problems in the Lorentz space when the data are in $L^{p^\prime}$, where $p$ is constant. Additionally, we use the techniques from \cite{betta2015neumannprob} to establish the existence of a renormalized solution to \eqref{5}. We deal with the following problems: on the one hand, the right-hand side involves functions in $L^1$; on the other hand, since $\Omega$ is unbounded, we cannot use compact embeddings of the form $W^{1,p}(\Omega; \vta) \hookrightarrow L^p(\Omega; \vta)$. 

We consider the following conditions:

\begin{enumerate}[label=($H_{\arabic*}$)]
\item \label{138} To simplify the representation, we will denote $\vta(x) = (1 + |x|)^{\alpha}$, $\textnormal{d}\vta = \vta  \textnormal{d}x$, and $\vta(F) = \int_F \vta  \textnormal{d}x$, where $F$ is a measurable set. Throughout the paper, 
$$
\Omega = \{x\in \mathbb{R} ^N \:|\: |(x_{d+1}, \ldots, x_N)|< \rho (|(x_1 , \ldots, x_d)|) \},
$$
where $\rho : [0,\infty) \rightarrow \mathbb{R}^+$ is $C^\infty$ with $|\rho^\prime| \leq C$ on $[0,\infty)$ for some positive constant, and $d \in \{1, \ldots, N-1\}$.

Additionally, we assume that $1<p<N$,  $\alpha, \gamma, \beta \in (-N , (p-1)N)$, $p-1<\alpha - \beta$, $0\leq\alpha - \gamma <p$,
$$
\vta (\Omega)<\infty, \quad \vtb (\partial \Omega)<\infty,
$$
$N+\alpha -p>0$, and $(\alpha-\gamma)N +\gamma p \geq 0$.

\item \label{136}    $A: \mathbb{R}^N \rightarrow \mathbb{R}^N$ and $K: \mathbb{R} \rightarrow \mathbb{R}$ are  continuous function satisfying:
\begin{gather*}
A(\xi)  \xi \geq \sigma |\xi|^p, \quad |A(\xi)| \leq \sigma ^{-1}|\xi|^{p-1}, \quad [A(\xi)-A(\eta)][ \xi-\eta]>0, \\
 |s|^{p-1} \leq K(s) \operatorname{sign} (s) \quad \text { and } \quad |K(s)| \leq \sigma ^{-1} |s|^{p-1} ,
\end{gather*}
for every  $\xi \in \mathbb{R}^N$, $\eta \in \mathbb{R}^N$, and $s\in \mathbb{R}$, where $0<\sigma <1$ and $\xi \neq \eta$.

\item  \label{137} Moreover, $H: \Omega \times \mathbb{R}^N \rightarrow \mathbb{R}$ and $G: \Omega \times \mathbb{R} \rightarrow \mathbb{R}$  are Carathéodory functions satisfying:
\begin{gather*}
| H(x, \xi) |\leq b(x)|\xi|^{p-1}, \quad b\in L^{\na , 1}(\Omega ; \vtao),\label{18}\\
G(x, s) s \geq 0, \quad |G(x,s)|\leq c(x) |s|^{r}, \quad c \in L^{z^{\prime}, 1}(\Omega ; \vtao), \label{19}
\end{gather*}
for almost every $x \in \Omega$ and for every $s \in \mathbb{R}$ and $\xi \in \mathbb{R}^N$, where $\na =  (N+\gamma)p/(p-\alpha + \gamma)$,
\begin{equation}\label{78}
0 \leq r<\frac{(N+\gamma) (p-1)}{N +\alpha -p}, \quad z=\frac{(N+\gamma)(p-1)}{N+\alpha -p} \frac{1}{r} \quad \text { and } \quad \frac{1}{z}+\frac{1}{z^{\prime}}=1 .
\end{equation}
\end{enumerate}

The main result of the paper is the following theorem: 
\begin{theorem} \label{146}
Assume that conditions \ref{138} - \ref{137} hold. There exists at least one renormalized solution (in the sense of definition \ref{172}) to problem \eqref{5}.
\end{theorem}

Additionally, using the arguments in the proof of Theorem \ref{146}, for a sequence $(u_n)$ of renormalized solutions of \eqref{142} satisfying suitable conditions, we have the following stability result.
\begin{theorem}\label{155} Under the assumptions \ref{138}, \ref{136}, and conditions \eqref{142} - \eqref{153} in Section \ref{167}, up to a subsequence (still indexed by $n$), $u_n$ converges to $u$, where $u$ is a renormalized solution to \eqref{5}, and
\begin{gather}
u_n \rightarrow u  \text { a.e in } \Omega, \text { and } u_n \rightarrow u   \text { a.e on } \partial \Omega,\\
A( \nabla T_k (u_n ) ) \nabla T_k (u_n) \rightharpoonup A( \nabla T_k(u)) \nabla T_k(u)  \quad \text {   in } L^1(\Omega ; \vta), \label{148}
\end{gather}
for all $k>0$, where $T_k$ is the truncation at level $k$ defined in \eqref{173}.
\end{theorem}



This paper is organized as follows. In Section \ref{16}, we gather preliminary definitions and results, which are used several times in the paper: renormalized solutions, embedding theorems, and equivalent norms in Sobolev spaces. In Section \ref{160}, we provide basic results for weak solutions and estimate the norms of these solutions in the Lorentz spaces. In Section \ref{161}, we prove Theorem \ref{146}, which asserts the existence of a renormalized solution to \eqref{5}. Its proof, contained herein, is based on the usual procedure of approximation, involving problems of the type \eqref{5} with smooth data that strongly converges to $f$ in $L^1$. For such a sequence of problems, weak solutions are obtained using fixed-point arguments (see Appendix \ref{139}). A priori estimates allow us to prove that these weak solutions converge, in some sense, to a function $u$, and a delicate procedure of passage to the limit allows us to establish that $u$ is a renormalized solution to \eqref{5}.

\section{Preliminary definitions and results} \label{16}

Before defining the type of solutions with which we will work, let's first establish some basic properties of weighted Sobolev and Lorentz spaces.

\subsection{A few properties of weighted Sobolev spaces}

Let $U\subset \mathbb{R}^{N}$ be an open set and  $p >1$, define
$$
L^{p}(U ; \vartheta _{\alpha})=\left\{u \:|\: u : \Omega \rightarrow \mathbb{R} \text { is  measurable function   and } \int_{U } |u|^{p} \textnormal{d}\vartheta _{\alpha} <\infty\right\},
$$
and
$$
W^{1, p  }(U  ; \vartheta _{\alpha} )=\left\{u \in L^{p  }\left(U   ; \vartheta _{\alpha}\right)\:|\:  |\nabla u| \in L^{p  }(U  ; \vartheta _{\alpha} )\right\},
$$
equipped with the norms
$$
\|u\|_{p  , U  , \alpha }=\left( \int_{U} |u|^{p} \textnormal{d}\vartheta _{\alpha}  \right)^{\frac{1}{p}} \quad \text { and } \quad  \|u\|_{1, p  , U  , \alpha }=\|u\|_{p  , U  , \alpha} + \|\nabla u\|_{p  , U , \alpha}, 
$$
respectively.  For details on these spaces see \cite{horiuchi1989imbedding, gurka1991continuous, pfluger1998compact, liu2008compact}.
\begin{proposition} \label{34}
(see \cite{pfluger1998compact, liu2008compact}) Suppose that $1 \leq p \leq q<\infty$, $\alpha, \beta \in (-N , (p-1)N)$, and  $0 \leq (N-1)/q-N/p+1 \leq \alpha / p - \beta / q$. Then, the trace operator $W^{1, p}(\Omega ; \vartheta_\alpha) \rightarrow L^q(\partial \Omega ; \vartheta_\beta)$ is continuous. Moreover, if $0 \leq (N-1)/q-N/p+1<\alpha / p - \beta / q$, then the trace operator is compact.
\end{proposition}

\begin{remark}
In condition \ref{138} (see \cite{pfluger1998compact, liu2008compact}), we can use a general condition related to the boundary $\partial \Omega$, assuming the existence of a locally finite covering of $\partial \Omega$ with open subsets $U_i \subset \mathbb{R}^N$ having the following properties:
\begin{enumerate}
\item[(a)] There is a global constant $\theta$ such that $\sum_i \chi_{U_i}(x) \leq \theta$ for every $x \in \mathbb{R}^N$;
\item[(b)] There exist cubes $B_i \subset \mathbb{R}^N$ and Lipschitz diffeomorphisms $\varphi_i: U_i \rightarrow B_i$ such that $0 \in B_i$ and $\varphi_i^{-1} (\mathbb{R}^{N-1} \times\{0\}) = U_i \cap \partial \Omega$;
\item[(c)] The partial derivatives of the coordinate functions $\varphi_i$ and $\varphi_i^{-1}$ are uniformly bounded by a constant $C$ (not depending on $i$).
\end{enumerate}
\end{remark}

Denote
$$
\mathcal{A} _n  =\sup  _{\|u\|_{1,p,\Omega,\alpha} \leq 1} \|u\|_{p,\Omega ^n , \alpha},
$$
where $\Omega ^n = \{x\in \mathbb{R} ^N \:|\: |x|> n\} \cap \Omega$, and $n\in \mathbb{N}$. Since 
$$
0\leq \mathcal{A} _{n+1} \leq \mathcal{A}_n \leq 1,
$$
the limit
$$
\mathcal{A} = \lim _{n\rightarrow \infty}\mathcal{A} _n
$$
exists, and $\mathcal{A} \in [0,1]$.

We have the following proposition:
\begin{proposition}\label{35}
Assume that $p \in(1, N)$, $\alpha, \beta \in (-N  , (p-1) N)$, and $p-1<\alpha - \beta$. For any $u \in W^{1, p}(\Omega ; \vartheta_\alpha)$, let
$$
\|u\|_{\partial}=\|u\|_{p, \partial \Omega, \beta}+\|\nabla u\|_{p, \Omega, \alpha} .
$$
Then $\|u\|_{\partial}$ is a norm on $W^{1, p}(\Omega ; \vartheta_\alpha)$ that is equivalent to
$$
\|u\|_{1, p, \Omega, \alpha}=\|u\|_{p, \Omega, \alpha}+\|\nabla u\|_{p, \Omega, \alpha} .
$$
\end{proposition}

Considering  that $\vta (\Omega)<\infty$ and $\vtb (\partial \Omega)<\infty$, Proposition \ref{35} is a consequence of Proposition \ref{34} and the following lemma.
\begin{lemma}
(see \cite[Lemma 3.2]{edmundsopic1993poincare}) Suppose $1 \leqslant p<\infty$. Let $F$ be a functional on $W^{1, p}(\Omega ; \vta)$ with the following properties:
\begin{enumerate}
\item[(i)] There is a constant $C_0$ such that for all $u, v \in W^{1, p}(\Omega ; \vta)$,
$$
|F(u)-F(v)| \leq  C_0\|u-v\|_{1, p, \Omega, \alpha} .
$$
\item[(ii)] $F(\lambda u)=\lambda F(u)$ for all $\lambda>0$ and  $u \in W^{1, p}(\Omega ; \vta)$.

\item[(iii)] If $u \in W^{1, p}(\Omega ; \vta)$ is constant in $\Omega$ and $F(u)=0$, then $u=0$.
\end{enumerate}

Let $\mathcal{A}<1$. Then there is a constant $C$ such that
$$
\int_{\Omega}|u|^p\dvta \leq C\left(|F(u)|^p+ \int_{\Omega} |\nabla u|^p  \dvta \right),
$$
for all $u \in W^{1, p}(\Omega ; \vta)$.
\end{lemma}

\begin{proposition}\label{36}
(see \cite[Theorem 3]{horiuchi1989imbedding}) Suppose that $p \in(1, N)$, $q \geq p$, $\alpha, \gamma \in (-N  , (p-1)N)$, $N/p- N/q \leq 1-\alpha / p + \gamma / q <N / p$ and $-N /q < \gamma / q \leq \alpha / p$. Then there is a continuous embedding $W^{1, p}(\Omega ; \vartheta_\alpha) \rightarrow L^q(\Omega ; \vartheta_{\gamma})$.
\end{proposition}

\subsection{A few properties of Lorentz spaces}
For $1<q<\infty$, the Lorentz space $L^{q, \infty}(\Omega ; \vtao)$ is the space of Lebesgue measurable functions such that
\begin{equation}\label{11}
\sup _{t>0} t^{1 / q} f^\ast _\gamma (t) <\infty.
\end{equation}
Here, $f^\ast$ denotes the decreasing rearrangement of $f$, i.e., the decreasing function defined by
\begin{equation}
  f^\ast _{\gamma} (t)=\inf \{s \geq 0 \:|\:  \vtao (\{x \in \Omega \:|\: |f(x)|>s \})  \leq t\}.
\end{equation}
In $L^{q,\infty} (\Omega ; \vtao)$, we work with the semi-norm
$$
\|f\|_{(\gamma , q ,\infty , \Omega)}= \sup _{t>0} t^{1 / q} f^\ast _{\gamma} (t).
$$

Furthermore, the Lorentz space $L^{q, 1}(\Omega;\vtao )$ is the space of Lebesgue measurable functions such that
\begin{equation}\label{12}
\|f\|_{(\gamma, q ,1, \Omega)}=\int_0^{\infty}t^{1 / q} f^\ast _{\gamma} (t) \frac{\textnormal{d} t}{t} <\infty,
\end{equation}
endowed with the norm defined by \eqref{12}.  For references about rearrangements, see \cite{castillo2021classicalmultidi}.

We will use the following properties of the Lorentz spaces, which are intermediate spaces between the Lebesgue spaces, in the sense that, for every $1<q_1<q_2<\infty$, the following relationships hold:
\begin{equation}\label{1}
L^{q_2 , 1}(\Omega ; \vtao) \subset L^{q_2, q_2}(\Omega ; \vtao )= L^{q_2} (\Omega ; \vtao ) \subset L^{q_2, \infty}(\Omega ; \vtao) \subset L^{q_1, 1}(\Omega ; \vtao) .
\end{equation}

One has the generalized Hölder inequality:
\begin{proposition}
For  any $ f \in L^{q, \infty}(\Omega ; \vtao )$ and   $g \in L^{q^{\prime}, 1}(\Omega ; \vtao)$, 
\begin{equation} \label{27}
\left|\int_{\Omega} f g \dvtao \right| \leq \|f\|_{(\gamma ,q,\infty , \Omega )}\|g\|_{(\gamma, q^\prime, 1 , \Omega)} .
\end{equation}
\end{proposition}

In what follows, for simplicity, we write
$$
 \|f\|_{(\gamma ,q,\infty , \Omega )} = \|f\|_{(\gamma ,q,\infty )}  \quad \text { and } \quad  \|f\|_{(\gamma, q^\prime, 1 , \Omega)} = \|f\|_{(\gamma, q^\prime, 1 )}. 
$$

\subsection{The definition of a renormalized solution}

Motivated by \cite{betta2002existenrenor, betta2015neumannprob}, we present a definition of renormalized solutions for nonlinear elliptic problems.

For $k > 0$, let $T_k: \mathbb{R} \rightarrow \mathbb{R}$ denote the standard truncation at level $k$, defined as
\begin{equation}\label{173}
T_k(s) = \left\{
\begin{aligned} 
& s & & \text{ if } |s| \leq k, \\ 
& k \operatorname{sign}(s)  & & \text{ if } |s| > k.
\end{aligned} \right.
\end{equation}

For $u \in W^{1, p}(\Omega; \vta)$,  we define (cf. \cite{andreu1997quasiellipandparab, benilan1995L1theory, andreuigbida2007existandunqi}):
$$
\mathcal{T}^{1, p}(\Omega; \vta) = \{u: \Omega \rightarrow \bar{\mathbb{R}} \text{ measurable such that } T_k(u) \in W^{1, p}(\Omega ; \vta ) \ \forall k > 0\}.
$$
Given $u \in \mathcal{T}^{1, p}(\Omega ; \vta)$, there exists a unique measurable function $v: \Omega \rightarrow \bar{\mathbb{R}}^N$ such that
\begin{equation}\label{152}
\nabla T_k(u) = v \chi_{\{|v| < k\}}, \quad \forall k > 0.
\end{equation}

This function $v$ is denoted as $\nabla u$. It is evident that if $u \in W^{1, p}(\Omega;\vta)$, then $v \in L^p(\Omega ;\vta)$, and $\nabla u = v$ in the usual sense.

As defined in \cite{andreu1997quasiellipandparab}, we denote $\mathcal{T}_{\textnormal{tr}}^{1, p}(\Omega;\vta)$ as the set of functions $u \in \mathcal{T}^{1, p}(\Omega ;\vta)$ for which there exists a sequence $(u_n)\subset W^{1, p}(\Omega ; \vta)$ satisfying the following conditions:
\begin{enumerate}
\item[(a)] $u_n$ converges to $u$ a.e. in $\Omega$.
\item[(b)] $\nabla T_k (u_n)\rightarrow \nabla T_k(u)$ in $L^1(V ;\vta)$ for any $k > 0$ and for any open bounded set $V\subset \Omega$.
\item[(c)] There exists a measurable function $w : \partial \Omega \rightarrow \bar{\mathbb{R}}$ such that $u_n$ converges to $w$ a.e. on $\partial \Omega$.
\end{enumerate}

The function $w$ is the trace of $u$ in a generalized sense. In subsequent discussions, the trace of $u \in \mathcal{T}_{\textnormal{tr}}^{1, p}(\Omega;\vta)$ on $\partial \Omega$ will be denoted by $\textnormal{tr}(u)$ or simply $u$. It is important to note that when $u \in W^{1, p}(\Omega;\vta)$, $\textnormal{tr}(u)$ coincides with the usual trace of $u$, denoted as $\tau(u)$. Additionally, for every $u \in \mathcal{T}_{\textnormal{tr}}^{1, p}(\Omega ; \vta)$ and every $k > 0$, $\tau(T_k(u)) = T_k(\textnormal{tr}(u))$. If $\phi \in W^{1, p}(\Omega ; \vta) \cap L^{\infty}(\Omega)$, then $u-\phi \in \mathcal{T}_{\textnormal{tr}}^{1, p}(\Omega ;\vta)$, and $\textnormal{tr}(u-\phi) = \textnormal{tr}(u) - \tau(\phi)$.

\begin{definition} \label{172}
Under \ref{138} - \ref{137} and $p < q \leq (N-1)p/(N-p)$, a function $u\in \mathcal{T}_{\textnormal{tr}}^{1, p}(\Omega;\vta)$  is a renormalized solution to \eqref{5} if $u$ is  finite  a.e. in $\Omega$, satisfying the following conditions:
\begin{gather}
T_k(u) \in W^{1, p}(\Omega ; \vta ), \quad \forall k > 0, \\
|u|^{p-1} \in L^{\past /p, \infty}(\Omega;\vtao) , \label{157}\\
|u|^{p-1} \in L^{q/p , \infty} (\partial \Omega ; \vtb).
\end{gather}
The gradient $\nabla u$ introduced in \eqref{152} satisfies:
\begin{gather}
|\nabla u|^{p-1} \in L^{\na^{\prime}, \infty}(\Omega;\vtao), \label{158}\\
\lim _{k \rightarrow \infty} \frac{1}{k} \int_{\{|u| < k\}} A( \nabla u)  \nabla u  \dvta = 0, \label{105}
\end{gather}
where $\past = (N+\gamma) p /(N+\alpha-p)$ and $1/\na + 1/\na ^\prime  = 1$.

Moreover, for every function $h \in W^{1, \infty}(\mathbb{R})$ with compact support and for every $\varphi \in$ $L^{\infty}(\Omega) \cap W^{1, p}(\Omega ; \vta)$, we have
\begin{equation}\label{104}
\begin{aligned}
&\int_{\Omega} h(u) A(\nabla u)  \nabla \varphi \dvta + \int_{\Omega} h^{\prime}(u) A(\nabla u) \nabla u \varphi \dvta  + \int_{\Omega} H(x,\nabla u) \varphi h(u) \dvtao  \\
&\qquad \quad \ \ + \int_{\Omega} G(x, u) \varphi h(u) \dvtao + \int_{ \partial \Omega} K(u) \varphi h(u) \dvtb \\
& \qquad = \int_{\Omega}  \varphi h(u) f \dvta + \int_{\partial \Omega}  \varphi h(u) g \dvtb.
\end{aligned}
\end{equation}
\end{definition}

\begin{remark}
Conditions \eqref{157} - \eqref{158} and \ref{137} on $H$, $G$, and $K$ ensure that for every renormalized solution,
$$
G(x, u), H(x, \nabla u) \in L^1(\Omega ; \vtao),  \text {  and } K(u) \in L^1(\partial \Omega ; \vtb).
$$
\end{remark}

\begin{remark}\label{174}
If $u$ is a renormalized solution of \eqref{5}, then $u$ is also a distributional solution satisfying
\begin{equation}\label{159}
\begin{aligned}
&\int_{\Omega}  A(\nabla u)  \nabla \varphi \dvta  + \int_{\Omega} H(x,\nabla u) \varphi  \dvtao  + \int_{\Omega} G(x, u) \varphi  \dvtao + \int_{ \partial \Omega} K(u) \varphi  \dvtb \\
&\qquad = \int_{\Omega}  \varphi  f \dvta + \int_{\partial \Omega}  \varphi  g \dvtb, \quad \forall \varphi \in C^\infty _c (\bar{\Omega}).
\end{aligned}
\end{equation}
Indeed, if $u$ is a renormalized solution of \eqref{5}, we know that $u$ is measurable and a.e. finite in $\Omega$, and that $T_k(u) \in W^{1, p}(\Omega ; \vta)$ for every $k>0$. This allows one to define $\nabla u$ in the sense of \eqref{152}. We also know that $|\nabla u|^{p-1} \in L^{\na ^{\prime}, \infty}(\Omega ; \vtao )$, so that $|A( \nabla u)| \in L^{\na^{\prime}, \infty}(\Omega ; \vtao)$ by the condition \ref{136}. Taking $\phi \in C_c ^{\infty}(\bar{\Omega})$ and $h_\ell$ defined by
$$
h_\ell (s)= \left\{
 \begin{aligned}
 & 0 & & \text { if } |s|>2 \ell, \\ 
& \frac{2 \ell-|s|}{\ell} &  & \text { if }  \ell <|s| \leq 2 \ell, \\ 
& 1 & & \text { if } |s| \leq \ell,
 \end{aligned}
 \right.
$$
and letting $\ell \rightarrow \infty$ in \eqref{104}, we obtain \eqref{159}.

Moreover, every renormalized solution $u$ of \eqref{5} belongs to $W^{1, q}(\Omega ; \vtao)$ for every $q<\na^{\prime}(p-1)$ when $p>2-1 / \na$. Indeed, $p>2-1 / \na$ implies $\na^{\prime}(p-1)>1$, and therefore, the gradient $\nabla u$ defined by \eqref{152}, which satisfies \eqref{158}, belongs to $(L^{q_1}(\Omega ; \vtao))^N$ for every $q_1<\na^{\prime}(p-1)$ and is the distributional gradient of $u$ (see \cite[Remark 2.10]{dalmasoorsinda1999renorm}).

\end{remark}

\section{Basic results for weak solutions} \label{160}

In this section, we assume more restrictive conditions on $f$, $g$, $H$, and $G$ to prove the existence of a renormalized solution $u$ to problem \eqref{5}.

Let $f_n \in L^{p^\prime} (\Omega ; \vta)$ and $g_n \in L^{p^\prime} (\partial \Omega ; \vtb)$ such that
\begin{gather}
f_n = 0 \text { on } \Omega \backslash \Omega _n, \quad g_n = 0 \text { on } \partial \Omega \backslash \partial \Omega _n,\\
f_n\rightarrow f \text { in } L^1 (\Omega ; \vta), \quad g_n\rightarrow g \text { in } L^1 (\partial\Omega ; \vtb),\\
\|f_n\|_{1,\Omega , \alpha}\leq  \|f\|_{1,\Omega , \alpha}, \forall n \quad \text { and } \quad \|g_n\|_{1,\partial \Omega , \beta}\leq  \|g\|_{1,\partial \Omega , \beta}, \forall n, \label{26}
\end{gather}
where  $\Omega _n =  \Omega \cap \{ |x| < n\}$ and $n\in \mathbb{N}$.

We set
\begin{equation*}
H_n(x, \xi)=T_n(H(x,  \xi)) \quad \text { and } \quad G_n(x, s)=T_n(G(x, s)).
\end{equation*}
Observe that
\begin{gather}
| H_n(x,  \xi)| \leq | H(x,  \xi) |\leq b(x)|\xi|^{p-1}, \quad H_n(x, \xi) \leq n, \label{24}\\
G_n(x, s) s \geq 0, \quad \left|G_n(x, s)\right| \leq|G(x, s)| \leq c(x)|s|^r, \quad \left|G_n(x, s)\right| \leq n. \label{25}
\end{gather}

Let $u_n \in W^{1, p}(\Omega ; \vta)$  be a weak solution of the following problem (see the Appendix):
\begin{equation}\label{7}
\left\{\begin{aligned}
-\operatorname{div}\left( \vta A(\nabla u_n)\right)+\vtao H_n(x,  \nabla u_n)+ \vta G_n (x, u_n) &= f _n & & \text { in } \Omega, \\
 \vta A(\nabla u_n) \cdot\nu + \vtb K(u_n)&= g_n   & & \text { on } \partial \Omega,
\end{aligned}\right.
\end{equation}
i.e., $u_n \in W^{1, p}(\Omega ; \vta )$, and
\begin{equation}\label{6}
\begin{aligned}
&\int_{\Omega} A(\nabla u_n)  \nabla v \dvta +\int_{\Omega}  H_n (x,  \nabla u_n) v \dvtao + \int_{\Omega} G_n(x, u_n) v \dvtao \\
& \qquad +  \int _{\partial \Omega} K(u_n) v \dvtb =\int_{\Omega} v f_n \dvta  + \int _{\partial \Omega} v g_n  \dvtb,
\end{aligned}
\end{equation}
for all $v\in W^{1,p} (\Omega ; \vta) $.

\begin{lemma} \label{42}
Let $u$ be a measurable function such that $T_k(u) \in W^{1, p}(\Omega ; \vta)$ for every $k>0$, and 
\begin{equation}\label{39}
\int_{\Omega}|\nabla T_k(u)|^p \dvta + \int_{\partial \Omega}|T_k(u)|^p \dvtb \leq C_0 k, \quad \forall k>0,
\end{equation}
where $C_0>0$ is a given constant. Then, $|u|^{p-1} \in L^{\past / p, \infty}(\Omega ; \vtao )$,  $|u|^{p-1} \in L^{q / p, \infty}(\partial \Omega ; \vtb )$, $|\nabla u|^{p-1} \in L^{\na^{\prime}, \infty}(\Omega ; \vtao)$, and
\begin{gather}
\||u|^{p-1}\|_{(\gamma , \past / p, \infty )} \leq C C_0 ,\label{37}\\
\||u|^{p-1}\|_{(\beta , q / p, \infty )} \leq C C_0 , \label{170} \\
\||\nabla u|^{p-1}\|_{(\gamma , \na^{\prime}, \infty)} \leq C C_0, \label{38}
\end{gather}
 where $C=C(N, p , q ,  \alpha , \gamma , \beta , \Omega)>0$, $p<q\leq (N-1)p/(N-p)$, and $\past = (N+\gamma)p/(N+\alpha - p)$.
\end{lemma}

\begin{proof}
First, we prove \eqref{37}. Using Propositions \ref{35} and \ref{36}, for every $k>0$, we have
\begin{equation} \label{40}
\begin{aligned}
k^{\past} \vtao( \{|u|>k\} ) \leq \int_{\Omega}|T_k(u)|^{\past} \dvtao \leq C\|T_k(u)\|_{1,p,\Omega,\alpha}^{\past} \leq C (C_0 k)^{\past / p},
\end{aligned}
\end{equation}
or equivalently, for every $s>0$,
$$
s^{\past /(p-1)} \vtao (\{ |u|^{p-1}>s\} )\leq C (C_0 s^{1 /(p-1)} )^{\past / p} .
$$
We deduce that
$$
\vtao( \{ |u|^{p-1}>s \}) \leq C(C_0 s^{-1})^{\past / p},
$$
and we get 
$$
(|u|^{p-1})^\ast _{\gamma} (t)\leq C_0(t/C)^{-p/\past}, \quad \forall t>0.
$$
Hence,
$$
\||u|^{p-1}\|_{(\gamma , \past / p , \infty)} \leq C C_0 ,
$$
which proves \eqref{37}. Applying Proposition \ref{34}, we can similarly deduce \eqref{170}.

Now we prove \eqref{38}. From \eqref{39}, we deduce that for every $s >0$ and  $k>0$
\begin{equation*}
 s ^p \vtao (\{ |\nabla u|>s, \ |u|<k\}) \leq \int_{\{|u|<k\}}|\nabla u|^p \dvta   =  \int_{\Omega}|\nabla T_k(u)|^p \dvta \leq C_0 k,
\end{equation*}
i.e., for every $s > 0$ and  $k>0$,
\begin{equation} \label{41}
s ^{p /(p-1)} \vtao ( \{ |\nabla u|^{p-1}>s, \ |u|<k \} \leq C_0 k .
\end{equation}
From \eqref{40} and \eqref{41}, we obtain that for every $s>0$ and  $k>0$,
\begin{align*}
&\vtao (\{   |\nabla u|^{p-1}>s \})\\ 
&\qquad \leq  \vtao ( \{ |\nabla u|^{p-1}>s , \ |u|<k \})  + \vtao ( \{  |\nabla u|^{p-1}>s , \ |u|>k \} ) \\
& \qquad \leq  \frac{C_0 k}{s ^{p^{\prime}}}+ \frac{C(C_0 k)^{\past / p}}{k^{\past}}=\frac{C_0 k}{s ^{p^{\prime}}} + CC_0^{\past / p} k^{-\past / p^\prime}.
\end{align*}

Choosing $k>0$ such that
$$
\frac{C_0 k}{s ^{p^{\prime}}}= C_0^{\past / p} k^{-\past / p^\prime},
$$
yields
$$
\vtao (\{x \in \Omega \:|\: |\nabla u|^{p-1}>s \}) \leq C\frac{C_0^{\na ^\prime}}{s^{\na ^\prime}}.
$$
Then
$$
(|\nabla u|^{p-1})^{\ast} _{\gamma} (t) \leq C\frac{C_0}{t^{1/\na ^\prime}} , \quad \forall t>0.
$$
Therefore, 
$$
\||\nabla u|^{p-1}\|_{(\gamma , \na ^\prime , \infty)} \leq C C_0,
$$
which proves \eqref{38}.\end{proof}

\begin{lemma}
Every solution $u_n$ of \eqref{7} satisfies
\begin{gather}
\||\nabla u_n|^{p-1}\|_{(\gamma, \na ^\prime, \infty)} \leq C, \label{21} \\
\||u_n|^{p-1}\|_{(\gamma, \past /p, \infty)} \leq C, \label{22}\\
\||u_n|^{p-1}\|_{(\beta , q /p, \infty)} \leq C, \label{171}
\end{gather}
where $p<q\leq (N-1)p/(N-p)$ and $C$ is a positive constant that depends only on  $p$, $q$, $\vta( \Omega)$, $\vtb( \Omega)$, $N$, $\alpha$, $\gamma$, $\beta$, $\|b\|_{(\gamma ,\na , 1)}$, $\|c\|_{(\gamma , z^\prime , 1)}$,  $\|f\|_{1 , \Omega , \alpha}$, and $\|g\|_{1 , \partial \Omega ,  \beta}$.
\end{lemma}
\begin{proof} We will proceed in two cases:\\

{\it First case: If  $\|b\|_{(\gamma , \na , 1)}$ is small enough.} \\

Using $T_k(u_n )$, $k>0$, as a test function in \eqref{6}, we obtain
\begin{equation}\label{23}
\begin{aligned}
&\int_{\Omega} A(\nabla u_n )  \nabla T_k(u_n ) \dvta +\int_{\Omega}  H_n (x,  \nabla u_n  ) T_k(u_n ) \dvtao + \int_{\Omega} G_n (x, u_n ) T_k(u_n ) \dvtao \\
& \qquad + \int _{\partial \Omega} K(u_n ) T_k(u_n ) \dvtb \leq \int_{\Omega} T_k(u_n ) f_n \dvta  + \int _{\partial \Omega} T_k(u_n ) g_n  \dvtb.
\end{aligned}
\end{equation}

From \ref{136}, we have
\begin{equation} \label{30}
\begin{aligned}
&\int_{\Omega} A(\nabla u_n )  \nabla T_k(u_n ) \dvta =\int_{\{u_n  \leq k\}} A( \nabla u_n )  \nabla u_n  \dvta \geq \sigma \int_{\{u_n  \leq k\}}|\nabla u_n |^p \dvta\\
&\qquad \geq \sigma \int_{\Omega}|\nabla T_k (u_n )|^p \dvta.
\end{aligned}
\end{equation}

On the other hand, by \eqref{24},  and using the generalized Hölder inequality \eqref{27} in the Lorentz spaces, we get
\begin{equation}\label{31}
\begin{aligned}
& \left|\int_{\Omega} H_n (x,  \nabla u_n ) T_k(u_n ) \dvtao \right|   \leq k \int_{\Omega} |H(x, \nabla u_n ) | \dvtao \leq k\int_{\Omega}  b|\nabla u_n |^{p-1}\dvtao  \\
& \qquad \leq k\|b\|_{(\gamma , \na , 1)} \||\nabla u_n |^{p-1}\|_{(\gamma , \na ^{\prime}, \infty)}.
\end{aligned}
\end{equation}
Using \eqref{25}, it results
\begin{equation}\label{32}
\int_{\Omega} G_n(x, u_n ) T_k(u_n ) \dvtao \geq 0 .
\end{equation}

Finally, we have
\begin{equation}\label{33}
\int_{\Omega}  T_k(u_n ) f_n \dvta \leq k\|f_n\|_{1,\Omega, \alpha} \quad \text { and } \quad \int_{\partial \Omega}  T_k(u_n) g_n \dvtb \leq k\|g_n\|_{1 , \partial \Omega , \beta }.
\end{equation}

Therefore, using \eqref{26} and \eqref{30} - \eqref{33}, we get
\begin{equation}\label{29}
\begin{aligned}
&\sigma \int_{\Omega} |\nabla T_k(u_n )|^p \dvta  + \sigma \int_{\partial \Omega} |T_k(u_n )|^p \dvtb \\
& \qquad  \leq k\left(\|b\|_{(\gamma , \na , 1)}\||\nabla u_n |^{p-1}\|_{(\gamma , \na ^{\prime}, \infty)} + \|f\|_{1,\Omega , \alpha} +   \|g\|_{1,\partial \Omega , \beta} \right).
\end{aligned}
\end{equation}
Let us define
\begin{equation}\label{28}
C_0= \frac{1}{\sigma} \left(\|b\|_{(\gamma , \na , 1)}\||\nabla u_n |^{p-1}\|_{(\gamma , \na^{\prime}, \infty)} +  \|f\|_{1,\Omega , \alpha} +   \|g\|_{1,\partial \Omega , \beta}  \right).
\end{equation}

Inequality \eqref{29} becomes
\begin{equation} \label{69}
\int_{\Omega}|\nabla T_k(u_n )|^p \dvta + \int_{ \partial \Omega}|T_k(u_n) |^p \dvtb  \leq C_0 k, \quad \forall k>0 .
\end{equation}
By Lemma \ref{42}, we get
$$
\||\nabla u_n |^{p-1}\|_{(\gamma, \na ^{\prime}, \infty)} \leq C C_0.
$$

If $\|b\|_{(\gamma , \na , 1)}$ is small enough,
$$
C \frac{1}{\sigma}\|b\|_{(\gamma , N, 1)}<\frac{1}{2},
$$
we obtain
$$
\||\nabla u_n|^{p-1}\|_{(\gamma , \na^{\prime}, \infty)} \leq  C,
$$
i.e., \eqref{21}.\\


{\it Second case: If  $\|b\|_{(\gamma , \na , 1)}$ can take any value.}\\

{\it Step 1.} As in \cite{betta2002existenrenor}, we define the following set $Z_n$. Since $\vtao(\Omega)$ is finite, the set $\mathcal{I}$ of the constants $c$ such that $\vtao (\{ |u_n |=c\})>0$ is at most countable. Let $\hat{Z}_{n}=\cup _{c\in \mathcal{I}} \{|u_n|=c\}$. Its complement  $Z_n= \Omega \backslash \hat{Z} _{n}$ is the union of sets such that  $\vtao (\{|u_{n}|=c\})=0$. Since, for every $c$,
\begin{equation}\label{51}
\nabla u_{n}=0 \quad \text { a.e. on } \{ |u_{n}|=c\},
\end{equation}
and since $\hat{Z}_{n}$ is at most a countable union, we obtain
\begin{equation}\label{45}
\nabla u_{n}=0 \quad \text { a.e. on } \hat{Z} _n.
\end{equation}

Since the constants $c$ such that  $\vtao (\{|u_n|=c\})>0$  have been eliminated by considering $Z_n$, it results that for $\ell_i$ fixed and $0<\ell<\ell_i$, the function
\begin{equation}\label{44}
\ell \in (0,\ell _i) \mapsto \vtao (Z_n \cap \{\ell<|u_n|< \ell_i \}) \text { is continuous}.
\end{equation}

Define for $\ell>0$, 
$$
S_\ell(s)= 
\left\{
\begin{aligned}
& 0 & & \text { if } |s| \leq \ell, \\ 
& (|s|-\ell) \operatorname{sign}(s) & &  \text { if }  |s|>\ell .
\end{aligned}
\right.
$$

Using in \eqref{6} the test function $T_k(S_\ell(u_n ))$ with $\ell$ to be specified later, we obtain
\begin{equation} \label{46}
\begin{aligned}
&\int_{\Omega} A(\nabla u_n )  \nabla T_k(S_\ell(u_n )) \dvta +\int_{\Omega}  H_n (x,  \nabla u_n  ) T_k(S_\ell(u_n )) \dvtao\\
& \qquad \quad \ \ + \int_{\Omega} G_n (x, u_n ) T_k(S_\ell(u_n )) \dvtao + \int _{\partial \Omega} K(u_n ) T_k(S_\ell(u_n )) \dvtb\\
& \qquad \leq \int_{\Omega} T_k(S_\ell(u_n )) f_n \dvta + \int _{\partial \Omega} T_k(S_\ell(u_n )) g_n  \dvtb.
\end{aligned}
\end{equation}

We have
\begin{equation}
\begin{aligned}
&\int_{\Omega} A(\nabla u_n ) \nabla T_k (S_\ell(u_n ))\dvta \geq \int_{\{\ell \leq u_n \leq \ell+k\}} A(\nabla u_n)  \nabla u_n \dvta \\
&\qquad \geq \sigma \int_{\Omega}  |\nabla T_k(S_\ell(u_n ))|^p \dvta ,
\end{aligned}
\end{equation}
\begin{equation}
\int_{\Omega} G_n (x,u_n )  T_k(S_\ell(u_n )) \dvtao \geq 0,
\end{equation}
\begin{equation}
\int_{\Omega}  T_k (S_\ell (u_n)) f_n \dvta \leq k\|f_n\|_{1, \Omega , \alpha}  \text { and }  \int_{\partial \Omega}  T_k (S_\ell (u_n)) g_n \dvtb \leq k\|g_n\|_{ 1 , \partial \Omega , \beta}.
\end{equation}

Let us now estimate
$$
\left| \int _{\Omega} H_n (x,\nabla u_n) T_k (S_\ell (u_n)) \dvtao\right|. 
$$ 
Using $S_\ell (s) =0$ for $|s|\leq \ell$,  \eqref{24}, and \eqref{45}, we have
\begin{align}
&\left|\int _{\Omega} H_n (x,\nabla u_n) T_k (S_{\ell} (u_n)) \dvtao \right|\leq k\int _{\{|u_n| > \ell\}} b|\nabla u_n|^{p-1}\dvtao\nonumber\\
& \qquad =k\int_{Z_n \cap \{|u_n|>\ell\}} b|\nabla S_\ell(u_n)|^{p-1} \dvtao \nonumber\\
& \qquad \leq k \|b\|_{(\gamma , \na , 1 , Z_n \cap\{u_n>\ell\})}\| |\nabla S_\ell (u_n)|^{p-1}\|_{(\gamma , \na ^\prime , \infty)}\label{43}
\end{align}

Combining \eqref{46} - \eqref{43} we have, for all $k>0$,
$$
\int_{\Omega}  |\nabla T_k(S_\ell(u_n ))|^p \dvta  + \int_{\partial \Omega}  | T_k(S_\ell(u_n ))|^p \dvtb  \leq \ell_1 k ,
$$
where $\ell_1$ is defined by
$$
\ell_1=\frac{1}{\sigma} \left(  \|b\|_{(\gamma , \na , 1 , Z_n \cap\{u_n>\ell\})}\| |\nabla S_\ell (u_n)|^{p-1}\|_{(\gamma , \na ^\prime , \infty)}+  \|f\|_{1,\Omega , \alpha} +   \|g\|_{1,\partial \Omega , \beta}  \right).
$$

By Lemma \ref{42}, we get
\begin{equation}\label{49}
\begin{aligned}
&\||\nabla S_\ell (u_n)| ^{p-1}\|_{(\gamma , \na ^\prime , \infty)}\\
&\ \ \leq  \frac{C}{\sigma} \left(  \|b\|_{(\gamma , \na , 1 , Z_n \cap\{u_n>\ell\})}\| |\nabla S_\ell (u_n)|^{p-1}\|_{(\gamma , \na ^\prime , \infty)}+  \|f\|_{1,\Omega , \alpha} +   \|g\|_{1,\partial \Omega , \beta}  \right)
\end{aligned}
\end{equation}

Since the decreasing rearrangement of the restrictions $\left.b\right|_{ E}$ and $b$ satisfies
\begin{equation}\label{54}
(\left.b \right|_{  E} )^\ast _{\gamma} (t) \leq b^\ast _{\gamma} (t), \quad t \in [0,\vtao ( E )],
\end{equation}
for any measurable set $E$, we have
\begin{equation}\label{55}
\begin{aligned}
& \|b\|_{(\gamma , \na , 1 , Z_n \cap\{u_n>\ell\})}  =\int_0^{\vtao (Z_n \cap \{|u_n|>\ell \})} (\left. b \right|  _{ Z_n \cap \{|u_n|>\ell\} } )^\ast  _{\gamma}(t) t^{1 / N} \frac{\dt}{t} \\
& \qquad  \leq \int_0^{\vtao ( Z_n \cap \{|u_n|>\ell\} ) }  b ^\ast _{\gamma}(t) t^{1 / N} \frac{\dt}{t} .
\end{aligned}
\end{equation}

In the case where
\begin{equation}\label{47}
\frac{C}{\sigma} \int_0^{\vtao (Z_n \cap \{|u_n|>0 \})} b^\ast _{\gamma}(t) t^{1 / N} \frac{\dt}{t} \leq \frac{1}{2},
\end{equation}
we choose $\ell=\ell_1=0$. If \eqref{47}  does not hold, we can choose $\ell=\ell_1>0$ such that
$$
\frac{C}{\sigma} \int_0^{\vtao (Z_n \cap \{|u_n|>\ell_1 \})} b^\ast _{\gamma}(t) t^{1 / N} \frac{\dt}{t} =\frac{1}{2} ;
$$
indeed, the function $\ell \rightarrow \vtao (Z_n \cap\{|u_n|>\ell \})$ is continuous (see \eqref{44}),  decreasing,  and tends to $0$ when $\ell$ tends to $\infty$. Note that $\ell_1$ actually depends on $n$ and $\gamma$.

Moreover, we define $\delta$ by
\begin{equation}\label{48}
\frac{C}{\sigma} \int_0^\delta   b^\ast _{\gamma}(t) t^{1 / N} \frac{\mathrm{d} t}{t}=\frac{1}{2},
\end{equation}
observe that $\delta$ does not depend on $n$. We have
\begin{equation}\label{63}
\vtao (Z_n \cap\{|u_n|>\ell_1\})=\delta .
\end{equation}

With this choice of $\ell=\ell_1$, we obtain from \eqref{49}
\begin{equation} \label{67}
\||\nabla S_{\ell_1} (u_n)| ^{p-1}\|_{(\gamma , \na ^\prime , \infty)}\\
\leq  \frac{2C}{\sigma} \left(  \|f\|_{1,\Omega , \alpha} +   \|g\|_{1,\partial \Omega , \beta}  \right).
\end{equation}\\


{\it Step 2.} Define, for $0 \leq \ell<\ell_1$, the function $S_{\ell, \ell_1}$ as
\begin{equation}
S_{\ell, \ell_1}(s)= \left\{
\begin{aligned}
&\ell_1-\ell & & \text { if }   s>\ell_1, \\ 
&s-\ell   & & \text { if }  \ell \leq s \leq \ell_1, \\ 
&0     & & \text { if }   -\ell \leq s \leq \ell, \\ 
&s+\ell   & & \text { if }  -\ell_1 \leq s \leq-\ell, \\ 
&\ell-\ell_1 & & \text { if }  s<-\ell_1 .
\end{aligned}
\right.
\end{equation}

Using in \eqref{6} the test function $T_k(S_{\ell, \ell_1}(u_n))$ with $\ell$ to be specified later, we obtain
\begin{equation}\label{53}
\begin{aligned}
&\int_{\Omega} A(\nabla u_n)  \nabla T_k(S_{\ell, \ell_1}(u_n)) \dvta +\int_{\Omega}  H_n (x,  \nabla u_n) T_k(S_{\ell, \ell_1}(u_n)) \dvtao \\
&\qquad \quad \ \ +\int_{\Omega} G_n(x, u_n) T_k(S_{\ell, \ell_1}(u_n)) \dvtao + \int _{\partial \Omega} K(u_n) T_k(S_{\ell, \ell_1}(u_n)) \dvtb \\
&\qquad =\int_{\Omega} T_k(S_{\ell, \ell_1}(u_n)) f_n \dvta  + \int _{\partial \Omega} T_k(S_{\ell, \ell_1}(u_n)) g_n  \dvtb,
\end{aligned}
\end{equation}
As in the previous step, we have
\begin{equation}
\int_{\Omega} A(\nabla u_n ) \nabla T_k (S_{\ell,\ell_1}(u_n ))\dvta\geq \sigma \int_{\Omega}  |\nabla T_k(S_{\ell,\ell_1} (u_n ))|^p \dvta ,
\end{equation}
\begin{equation}
\int_{\Omega} G_n (x,u_n )  T_k(S_{\ell,\ell_1} (u_n )) \dvtao \geq 0,
\end{equation}
\begin{equation}
\int_{\Omega}  T_k (S_{\ell,\ell_1} (u_n)) f_n \dvta \leq k\|f\|_{1 , \Omega  , \alpha}  \text { and }  \int_{\partial \Omega}  T_k (S_{\ell,\ell_1} (u_n)) g_n \dvtb \leq k\|g\|_{1 , \partial \Omega , \beta}.
\end{equation}
 
Moreover, using $S_{\ell, \ell_1}(s)=0$ for $|s| \leq \ell$ and \eqref{24}, we have
\begin{equation}\label{50}
\begin{aligned}
& \left|\int_{\Omega} H_n(x, \nabla u_n) T_k(S_{\ell, \ell_1}(u_n)) \dvtao \right|  \leq k \int_{\left\{\left|u_n\right|>\ell\right\}} b|\nabla u_n|^{p-1}\dvtao  \\
& \qquad \leq k\left( \int_{\{\ell<|u_n|<\ell_1\}} b|\nabla u_n|^{p-1} \dvtao +\int_{\{|u_n| \geq \ell_1\}} b |\nabla u_n|^{p-1}\dvtao \right) .
\end{aligned}
\end{equation}
Let us estimate each term of the right-hand side of \eqref{50}. Using property \eqref{51} of $Z_n$ and the generalized Hölder inequality  in the Lorentz spaces, we have
\begin{equation}
\begin{aligned}
& \int_{\{\ell<|u_n|<\ell_1 \}} b |\nabla u_n|^{p-1} \dvtao  =\int_{Z_n \cap \{\ell<|u_n|<\ell_1 \} } b|\nabla S_{\ell,\ell_1}(u_n)|^{p-1} \dvtao \\
& \qquad \leq \|b\|_{(\gamma , \na , 1 , Z_n \cap\{\ell<|u_n|<\ell_1 \})}\| |\nabla S_{\ell,\ell_1} (u_n)|^{p-1}\|_{(\gamma , \na ^\prime , \infty)}.
\end{aligned}
\end{equation}

Similarly, for the second term of the right-hand side of \eqref{50} we have:
\begin{equation*}
\begin{aligned}
& \int_{\{ |u_n|\geq \ell_1 \}} b |\nabla u_n|^{p-1} \dvtao  =\int_{\Omega} b|\nabla S_{\ell_1}(u_n)|^{p-1} \dvtao \\
& \qquad \leq \|b\|_{(\gamma , \na , 1 )}\| |\nabla S_{\ell_1} (u_n)|^{p-1}\|_{(\gamma , \na ^\prime , \infty)}.
\end{aligned}
\end{equation*}

Therefore, we have:
\begin{equation}\label{52}
\begin{aligned}
& \left|\int_{\Omega} H_n (x, \nabla u_n ) T_k (S_{\ell, \ell_1} (u_n))\dvtao \right| \\
& \qquad \leq k\left( \|b\|_{(\gamma , \na , 1 , Z_n \cap\{\ell<|u_n|<\ell_1 \})}\| |\nabla S_{\ell,\ell_1} (u_n)|^{p-1}\|_{(\gamma , \na ^\prime , \infty)} \right. \\
& \qquad \quad \ \ + \left.\|b\|_{(\gamma , \na , 1 )}\| |\nabla S_{\ell_1} (u_n)|^{p-1}\|_{(\gamma , \na ^\prime , \infty)}\right) .
\end{aligned}
\end{equation}

Combining \eqref{53} - \eqref{52} we have, for all $k>0$
$$
\int_{\Omega}  |\nabla T_k(S_{\ell,\ell_1}(u_n ))|^p \dvta  + \int_{\partial \Omega}  | T_k(S_{\ell,\ell_1}(u_n ))|^p \dvtb \leq \ell_2 k ,
$$
where $\ell_2$ is defined by
$$
\begin{aligned}
&\ell_2=  \frac{1}{\sigma} \left(\|b\|_{(\gamma , \na , 1 , Z_n \cap\{\ell<|u_n|<\ell_1 \})}\| |\nabla S_{\ell,\ell_1} (u_n)|^{p-1}\|_{(\gamma , \na ^\prime , \infty)} \right. \\
& \qquad + \left.\|b\|_{(\gamma , \na , 1 )}\| |\nabla S_{\ell_1} (u_n)|^{p-1}\|_{(\gamma , \na ^\prime , \infty)} +  \|f\|_{1,\Omega , \alpha} +   \|g\|_{1,\partial \Omega , \beta} \right).
\end{aligned}
$$

By Lemma \ref{42}, we get:
\begin{equation}\label{57}
\begin{aligned}
& \||\nabla S_{\ell, \ell_1} (u_n )|^{p-1}\|_{(\gamma , \na ^{\prime} , \infty )} \\
& \qquad \leq \frac{C}{\sigma} \left(\|b\|_{(\gamma , \na , 1 , Z_n \cap\{\ell<|u_n|<\ell_1 \})}\| |\nabla S_{\ell,\ell_1} (u_n)|^{p-1}\|_{(\gamma , \na ^\prime , \infty)} \right. \\
& \qquad \quad \ \ + \left.\|b\|_{(\gamma , \na , 1 )}\| |\nabla S_{\ell_1} (u_n)|^{p-1}\|_{(\gamma , \na ^\prime , \infty)} +  \|f\|_{1,\Omega , \alpha} +   \|g\|_{1,\partial \Omega , \beta} \right).
\end{aligned}
\end{equation}

Similarly to \eqref{55}, using \eqref{54}, we have
\begin{equation}\label{60}
\|b\|_{(\gamma, \na, 1 , Z_n \cap\{\ell<|u_n|<\ell_1\})} \leq \int_0^{\vtao ( Z_n \cap \{\ell< |u_n |<\ell_1\})}b^\ast _{\gamma}(t) t^{1 / N} \frac{\dt}{t} .
\end{equation}

In the case where
\begin{equation}\label{56}
 \frac{C}{\sigma} \int_0^{\vtao ( Z_n \cap \{0< |u_n |<\ell_1 \}) } b^\ast _{\gamma}(t) t^{1 / N} \frac{\dt}{t} \leq \frac{1}{2},
\end{equation}
we choose $\ell=\ell_2=0$. If \eqref{56} does not hold, since  the function $\ell \rightarrow \vtao (Z_n \cap \{\ell< |u_n |<\ell_1\})$ is continuous, we can choose $\ell=\ell_2>0$ such that
$$
 \frac{C}{\sigma} \int_0^{\vtao ( Z_n \cap \{\ell_2< |u_n |<\ell_1 \}) } b^\ast _{\gamma}(t) t^{1 / N} \frac{\dt}{t} = \frac{1}{2}.
$$
Note that $\ell_2$ actually depends on $n$ and $\gamma$. We have
\begin{equation}\label{64}
\left|Z_n \cap\left\{\ell_2<\left|u_n\right|<\ell_1\right\}\right|=\delta,
\end{equation}
where $\delta$ is defined by \eqref{48}. With this choice of $\ell=\ell_2$, we obtain from \eqref{57}
\begin{equation}\label{68}
\begin{aligned}
& \||\nabla S_{\ell_2, \ell_1}(u_n)|^{p-1}\|_{(\gamma , \na ^{\prime}, \infty)} \\
& \qquad \leq \frac{2C}{\sigma} \left(\|b\|_{(\gamma , \na , 1 )}\| |\nabla S_{\ell_1} (u_n)|^{p-1}\|_{(\gamma , \na ^\prime , \infty)} +  \|f\|_{1,\Omega , \alpha} +   \|g\|_{1,\partial \Omega , \beta} \right).
\end{aligned}
\end{equation}\\


{\it Step 3.} Define for $0 \leq \ell<\ell_2$ the function $S_{\ell, \ell_2}$ :
$$
S_{\ell, \ell_2}(s)= \left\{ 
\begin{aligned}
& \ell_2-\ell & & \text { if }   s>\ell_2, \\ 
& s-\ell & & \text { if }   \ell \leq s \leq \ell_2, \\ 
& 0 & & \text { if }  -\ell \leq s \leq \ell, \\ 
& s+\ell & & \text { if }   -\ell_2 \leq s \leq-\ell, \\
& \ell-\ell_2 & & \text { if }  s<-\ell_2,
\end{aligned}\right.
$$
for every $s \in \mathbb{R}$. Using in \eqref{6} the test function $T_k\left(S_{\ell, \ell_2}\left(u_n\right)\right)$ with $\ell$ to be specified later, we obtain
\begin{equation}\label{58}
\begin{aligned}
&\int_{\Omega} A(\nabla u_n)  \nabla T_k(S_{\ell, \ell_2}(u_n)) \dvta +\int_{\Omega}  H_n (x,  \nabla u_n) T_k(S_{\ell, \ell_2}(u_n)) \dvtao \\
&\qquad \quad \ \ +\int_{\Omega} G_n(x, u_n) T_k(S_{\ell, \ell_2}(u_n)) \dvtao + \int _{\partial \Omega} K(u_n) T_k(S_{\ell, \ell_2}(u_n)) \dvtb \\
&\qquad =\int_{\Omega} T_k(S_{\ell, \ell_2}(u_n)) f_n \dvta  + \int _{\partial \Omega} T_k(S_{\ell, \ell_2}(u_n)) g_n  \dvtb,
\end{aligned}
\end{equation}

As in \eqref{50} - \eqref{52}, we have
\begin{align*}
&\left|\int_{\Omega} H_n (x,  \nabla u_n) T_k (S_{\ell, \ell_2}(u_n))\dvtao \right|\\
& \qquad \leq k\left(\int_{\{\ell<|u_n|<\ell_2\}} b|\nabla u_n|^{p-1} \dvtao +\int_{\{\ell_2<|u_n|<\ell_1 \}} b|\nabla u_n|^{p-1}\dvtao \right. \\
&\qquad \quad \ \ \left. +\int_{\{|u_n| \geq \ell_1\}} b|\nabla u_n |^{p-1} \dvtao \right) \\
& \qquad \leq k\left( \|b\|_{(\gamma , \na , 1 , Z_n \cap \{\ell<|u_n|<\ell_2\})}  \| |\nabla S_{\ell, \ell_2} (u_n )|^{p-1}\|_{(\gamma , \na ^{\prime}, \infty)}\right. \\
& \qquad \quad \ \ +\|b\|_{(\gamma , \na, 1)} \| |\nabla S_{\ell_2, \ell_1} (u_n) |^{p-1}\|_{(\gamma , \na^{\prime}, \infty)} \\
& \qquad \quad \ \ \left.+\|b\|_{(\gamma, \na , 1)}  \| |\nabla S_{\ell_1} (u_n)|^{p-1}\|_{(\gamma , \na^{\prime}, \infty )}\right).
\end{align*}

As in \eqref{57}, we deduce
\begin{equation}\label{59}
\begin{aligned}
& \||\nabla S_{\ell, \ell_2} (u_n )|^{p-1}\|_{(\gamma , \na ^{\prime} , \infty )} \\
& \qquad \leq \frac{C}{\sigma} \left( \|b\|_{(\gamma , \na , 1 , Z_n \cap\{\ell<|u_n|<\ell_2 \})}\| |\nabla S_{\ell,\ell_2} (u_n)|^{p-1}\|_{(\gamma , \na ^\prime , \infty)} \right. \\
& \qquad \quad \ \ +  \|b\|_{(\gamma , \na , 1 )} \| |\nabla S_{\ell_2,\ell_1} (u_n)|^{p-1}\|_{(\gamma , \na ^\prime , \infty)}\\
& \qquad \quad \ \ + \left.\|b\|_{(\gamma , \na , 1 )}\| |\nabla S_{\ell_1} (u_n)|^{p-1}\|_{(\gamma , \na ^\prime , \infty)} +  \|f\|_{1,\Omega , \alpha} +   \|g\|_{1,\partial \Omega , \beta} \right).
\end{aligned}
\end{equation}
Arguing as in \eqref{60},
$$
\|b\|_{(\gamma , \na , 1,(Z_n \cap\{\ell<|u_n|<\ell_2\})} \leq \int_0^{\vtao (Z_n \cap \{\ell<|u_n|<\ell_2\})} b^\ast _{\gamma}(t) t^{1 / N} \frac{\dt}{t} .
$$

In the case where
\begin{equation}\label{61}
\frac{C}{\sigma} \int_0^{\vtao ( Z_n \cap \{0<|u_n|<\ell_2\})} b^\ast _{\gamma}(t) t^{1 / N} \frac{\dt}{t} \leq \frac{1}{2},
\end{equation}
we choose $\ell=\ell_3=0$. If \eqref{61} does not hold, we can choose $\ell=\ell_3>0$ such that
$$
\frac{C}{\sigma} \int_0^{\vtao( Z_n \cap \{\ell_3<|u_n|<\ell_2\})} b^\ast _{\gamma}(t) t^{1 / N} \frac{\dt}{t}=\frac{1}{2} .
$$
Note that $\ell_3$ actually depends on $n$ and $\gamma$. We have
\begin{equation}\label{65}
\vtao ( Z_n \cap \{\ell_3<|u_n|<\ell_2\})=\delta,
\end{equation}
where $\delta$ is defined by \eqref{48}. With this choice of $\ell=\ell_3$, we obtain from \eqref{59}
\begin{equation}\label{62}
\begin{aligned}
& \||\nabla S_{\ell_3, \ell_2} (u_n )|^{p-1}\|_{(\gamma , \na ^{\prime} , \infty )} \\
& \qquad \leq \frac{2C}{\sigma} \left( \|b\|_{(\gamma , \na , 1 )} \| |\nabla S_{\ell_2,\ell_1} (u_n)|^{p-1}\|_{(\gamma , \na ^\prime , \infty)}\right.\\
& \qquad \quad \ \  + \left.\|b\|_{(\gamma , \na , 1 )}\| |\nabla S_{\ell_1} (u_n)|^{p-1}\|_{(\gamma , \na ^\prime , \infty)} +  \|f\|_{1,\Omega , \alpha} +   \|g\|_{1,\partial \Omega , \beta} \right).
\end{aligned}
\end{equation}\\


{\it Steep 4.} We repeat this procedure until the time it stops, i.e., when we arrive to some $i=I$ (which depends on $n$ and $\gamma$) for which we have
$$
\frac{C}{\sigma} \int_0^{\vtao ( Z_n \cap \{0< |u_n |<\ell_{I-1}\})} b^\ast _{\gamma}(t) t^{1 / N} \frac{\dt}{t} \leq \frac{1}{2},
$$
then we choose
$$
\ell_I=0.
$$
Let us now estimate $I$. We have
$$
\begin{aligned}
& \vtao ( \Omega) \geq \vtao ( Z_n )  \geq  \vtao ( Z_n \cap\{|u_n|>\ell_1\})+ \vtao (Z_n \cap \{\ell_2<|u_n|<\ell_1\}) \\
&\qquad  +\vtao (Z_n \cap \{\ell_3<|u_n|<\ell_2\}) + \cdots + \vtao ( Z_n \cap \{\ell_{I-1}<|u_n|<\ell_{I-2}\})
\end{aligned}
$$
and, in view of \eqref{63}, \eqref{64}, and \eqref{65} we know that
$$
\begin{aligned}
& \vtao ( Z_n \cap \{|u_n|>\ell_1\})  =\vtao ( Z_n \cap  \{\ell_2< |u_n |<\ell_1 \} )= \cdots \\
& \qquad = \vtao ( Z_n \cap \{\ell_{I-1}< |u_n |<\ell_{I-2} \}) =\delta,
\end{aligned}
$$
where $\delta$ is defined by \eqref{48}, and does not depend on $n$. Therefore, $(I-1) \delta \leq \vtao ( \Omega )$, and
$$
I \leq I^*=1+\left[\frac{\vtao (\Omega)}{\delta}\right],
$$
where $[s]$ denotes the integer part of $s$.

Observe that $I$ is estimated by the number $I^*$ which does not depend on $n$, and which depends on $b^\ast _\gamma$ and  $\delta$. We define
\begin{equation}\label{66}
\ell_0=\infty, \quad S_{\ell_1, \ell_0}=S_{\ell_1},
\end{equation}
and defining
\begin{align*}
&X_i =\||\nabla S_{\ell_i, \ell_{i-1}} (u_n)|^{p-1}\|_{(\gamma ,\na^{\prime}, \infty)}, \quad \text { for } 1 \leq i \leq I ,\\
&a_1 = \frac{2C}{\sigma} \|b\|_{(\gamma , \na , 1)}, \\
&a_2 = \frac{2C}{\sigma} \left( \|f\|_{1,\Omega , \alpha} +   \|g\|_{1,\partial \Omega , \beta} \right).
\end{align*}
Observe that
$$
X_1=\||\nabla S_{\ell_1, \ell_0}(u_n)|^{p-1}\|_{(\gamma , \na ^{\prime}, \infty)} = \||\nabla S_{\ell_1}(u_n)|^{p-1}\|_{(\gamma , \na ^{\prime}, \infty)} .
$$

From \eqref{67}, \eqref{68}, \eqref{62}, and \eqref{66}, we see
\begin{align*}
& X_1 \leq a_2, \quad X_2 \leq a_1 X_1+a_2, \quad  X_3 \leq a_1 X_2 + a_1  X_1+a_2, \quad \ldots, \\
& X_I \leq a_1 X_{I-1}+\cdots+a_1 X_1+a_2, \quad I \leq I^* .
\end{align*}

It can be proved by induction that
$$
X_i \leq (a_1 + 1)^{i-1} a_2 \quad \text { for } 1 \leq i \leq I .
$$
Since $\ell_I=0$, we have
$$
|\nabla u_n|^{p-1}=\sum_{i=1}^I |\nabla u_n|^{p-1} \chi_{\{\ell_i<|u_n|<\ell_{i-1}\}}=\sum_{i=1}^I |\nabla S_{\ell_i, \ell_{i-1}} (u_n)|^{p-1}.
$$
Therefore, 
\begin{align*}
&\| |\nabla u_n |^{p-1} \|_{( \gamma , \na^{\prime}, \infty )}  \leq \sum_{i=1}^I \| |\nabla S_{\ell_i, \ell_{i-1}} (u_n ) |^{p-1} \|_{(\gamma , \na ^{\prime}, \infty)} \\
& \qquad \leq \sum_{i=1}^I X_i  \leq a_2 \sum_{i=1}^I (a_1+1)^{i-1}\\
& \qquad =a_2 \frac{(a_1+1)^I-1}{a_1} \leq \frac{a_2}{a_1}[(a_1+1)^{I^\ast}-1],
\end{align*}
i.e., the desired result, \eqref{21}.

Let us finally prove the result \eqref{22}. From \eqref{69}, we have 
$$
\int_{\Omega} |\nabla T_k (u_n)|^p \dvta + \int_{\partial \Omega} |T_k (u_n)|^p \dvtb  \leq C_0 k, \quad \forall k>0,
$$
where the constant $C_0$ defined by \eqref{28} is now bounded independently on $n$ in view of \eqref{21}. The results \eqref{22} and \eqref{171} then follow from \eqref{37} and \eqref{170}, respectively. \end{proof}


\section{Existence results for renormalized solutions} \label{161}

In this section, we prove our main results, which give the existence of a renormalized solution to problem \eqref{5}.

\renewcommand*{\proofname}{{\bf Proof of Theorem \ref{146}}}

\begin{proof} $ $\\

{\it Step 1.  A priori estimates}. \\

Using $T_k (u_n)$ for $k>0$, as test function in \eqref{6} we have
\begin{equation*}
\begin{aligned}
&\int_{\Omega} A(\nabla u_n)  \nabla T_k (u_n) \dvta +\int_{\Omega}  H_n(x,  \nabla u_n) T_k (u_n) \dvtao \\
& \qquad \quad \ \ + \int_{\Omega} G_n\left(x, u_n\right) T_k (u_n) \dvtao  + \int _{\partial \Omega} K(u_n) T_k (u_n) \dvtb \\
& \qquad =\int_{\Omega} T_k (u_n) f_n \dvta  + \int _{\partial \Omega} T_k (u_n) g_n  \dvtb,
\end{aligned}
\end{equation*}
which implies, by \ref{136} and \ref{137},

\begin{align*}
&\sigma \int_{\Omega}  \left|\nabla T_k (u_n)\right|^p \dvta +  \int _{\partial \Omega} |T_k (u_n)|^p \dvtb\\
& \qquad \leq  k\int_{\Omega} b|\nabla u_n|^{p-1} \dvtao + k\|f_n\|_{1, \Omega , \alpha} + k \|g_n\|_{ 1, \partial \Omega , \beta} .
\end{align*}

By \eqref{21} and  \eqref{26} we get
\begin{equation}\label{70}
\int_{\Omega} |\nabla T_k (u_{n})|^p \dvta + \int _{\partial \Omega} |T_k (u_n)|^p \dvtb\leq C k
\end{equation}
for a suitable positive constant $C$  does not depend on $k$ and $n$. We deduce that, for every $k>0$,
\begin{equation}\label{96}
T_k (u_n ) \text { is bounded in } W^{1, p}(\Omega ; \vta).
\end{equation}

Moreover taking into account \ref{136} and \eqref{70}, we obtain that for any $k>0$
\begin{equation}\label{99}
A(\nabla T_k (u_n)) \text { is bounded in } (L^{p^{\prime}}(\Omega ; \vta))^N ,
\end{equation}
uniformly with respect to $n$.\\


{\it Step 2.  We prove}
\begin{equation}\label{8}
\begin{gathered} 
 \vta (\{x \in \Omega \:|\: |u_n(x)| > L \}) \leq \frac{C}{\ln (1+L)},\\   
 \vtb (\{x \in \partial \Omega \:|\: |u_n(x)| > L \}) \leq \frac{C}{\ln (1+L)},
 \end{gathered}
\end{equation}
{\it for all $n$, where $C>0$ is a constant independent of $n$.}\\

For $q\geq 1$, we consider the function
$$
\Psi_q(r)=\int_0^r \frac{1}{(1+|t|)^q} \dt, \quad \forall r \in \mathbb{R}.
$$
For $q=1$, we have
\begin{equation}\label{75}
\Psi_1(r)=\int_0^{r} \frac{1}{1+|t|} \dt=\operatorname{sign}(r) \ln (1+|r|), 
\end{equation}
and by H\"older inequality,
\begin{equation}\label{76}
|\Psi_1(r)|^p= \left| \int_0^{r} \frac{1}{1+|t|} \dt \right|^p \leq |r|^{p/p^\prime} \left|\int_0^{r} \frac{1}{(1+|t|)^p} \dt \right| = |r|^{p-1} |\Psi_p (r)|  .
\end{equation}

Using $\Psi_p (u_n)$  as a test function in \eqref{6}, we get
\begin{equation}\label{74}
\begin{aligned}
&\int_{\Omega} A(\nabla u_n)  \nabla \Psi_p (u_n)  \dvta +\int_{\Omega}  H_n (x,  \nabla u_n) \Psi_p (u_n) \dvtao \\
&\qquad \quad \ \ + \int_{\Omega} G_n(x, u_n) \Psi_p (u_n) \dvtao  + \int _{\partial \Omega} K(u_n) \Psi_p (u_n) \dvtb \\
& \qquad =\int_{\Omega} \Psi_p (u_n) f_n \dvta  + \int _{\partial \Omega} \Psi_p (u_n) g_n  \dvtb.
\end{aligned}
\end{equation}

By \ref{136}, \eqref{24}, \eqref{25}, and since $\|\Psi_p(u_n)\|_{L^{\infty}(\Omega)} \leq \frac{1}{p-1}$, we get
\begin{align*}
&\sigma \int_{\Omega} \frac{|\nabla u_n |^p}{ (1+|u_n|)^p} \dvta + \int _{\partial \Omega} K(u_n) \Psi _p (u_n) \dvtb  \\
& \qquad \leq \frac{1}{p-1} \int_{\Omega} b |\nabla u_n| ^{p-1} \dvtao + \frac{1}{p-1}\|f_n\|_{1 ,\Omega , \alpha } + \frac{1}{p-1}\|g_n\|_{ 1 , \partial \Omega , \beta}.
\end{align*}
By \ref{136}, \eqref{26}, \eqref{21}, \eqref{75}, and \eqref{76}, we have
$$
\int_{\Omega} |\nabla \Psi _1 (u_n)|^p \dvta + \int _{\partial \Omega} |\Psi _1 (u_n)|^{p}  \dvtb  \leq C, \quad \forall n,
$$
where $C>0$ is a constant independent of $n$.

By Propositions \ref{35} and \ref{36}, we have
$$
\|\Psi_1 (u_n)\|_{p , \Omega , \alpha } \leq C.
$$
Hence, according to the definition of $\Psi_1$, we obtain  \eqref{8}.

Furthermore, as a consequence of \eqref{8}, we make the following assertions:

By \eqref{96}, compact embedding theorem, and Proposition \ref{34}, we can assume that
\begin{gather}
T_k (u_n )|_{\Omega _m}  \text { is a Cauchy sequence in measure}, \label{97}\\ 
T_k ( \tau ( u_n ) ) = \tau (T_k (u_n))  \text { is a Cauchy sequence in measure}, \label{169}
\end{gather}
where $\Omega _m = \Omega \cap \{|x|< m\}$.

Let $\varepsilon > 0$ be fixed. For every $\ell > 0$ and every $n_1$ and $n_2$, we have
\begin{gather}
\{ |u_{n_1} - u_{n_2} |> \varepsilon\} \subset \{ |u_{n_1} | >\ell \} \cup  \{ |u_{n_2} | >\ell \} \cup \{ |T_\ell (u_{n_1}) - T_\ell (u_{n_2}) |> \varepsilon\}. \label{98}
\end{gather}

Combining \eqref{98}, \eqref{8},  \eqref{97}, and \eqref{169}, we conclude that
\begin{gather*}
u_n |_{\Omega _m}  \text { and } \tau( u_n ) \text { are Cauchy sequences in measure}.
\end{gather*}

Therefore, there exist measurable functions  $u: \Omega \rightarrow \bar{\mathbb{R}}$ and $w: \partial\Omega \rightarrow \bar{\mathbb{R}}$ finites a.e. in $\Omega$. Additionally, according to  \eqref{96} and \eqref{99}, for any $k>0$ there exists a function $V_k \in (L^{p^{\prime}}(\Omega ; \vta))^N$ such that, up to a subsequence still indexed by $n$,
\begin{gather}
u_n \rightarrow u \quad  \text {  a.e. in } \Omega , \label{73}\\
T_k (u_{n}) \rightharpoonup T_k(u) \quad  \text {  in } W^{1, p}(\Omega ; \vta), \label{71}\\
\tau (T_k(u))= T_k (w), \\
A(\nabla T_k(u_n)) \rightharpoonup V _k \quad \text { in } (L^{p^{\prime}}(\Omega ; \vta ))^N, \label{72} 
\end{gather}
for all $k>0$. Furthermore, since $W^{1, p} (\Omega ; \vartheta_\alpha ) \hookrightarrow L^p (\partial \Omega ; \vartheta_\beta)$, we can assume that
\begin{equation}\label{95}
T_k (u_{n}) \rightarrow T_k(w), \quad \text {  in } L^{ p}(\partial \Omega ; \vtb).
\end{equation}\\


{\it Step 3. The following limits holds.}
\begin{equation} \label{80}
\lim _{k \rightarrow \infty} \limsup _{n \rightarrow \infty} \frac{1}{k} \int_{\Omega} A (\nabla u_n) \nabla T_k (u_n) \dvta =0.
\end{equation}\\

Using the test function $\frac{1}{k} T_k (u_n)$ in \eqref{6}, we have
\begin{equation*}
\begin{aligned}
&\frac{1}{k} \int_{\Omega} A(\nabla u_n)  \nabla T_k (u_n)  \dvta + \frac{1}{k}\int_{\Omega}  H_n (x,  \nabla u_n) T_k (u_n) \dvtao \\
& \qquad \quad \ \ + \frac{1}{k}\int_{\Omega} G_n(x, u_n) T_k (u_n) \dvtao   + \frac{1}{k}\int _{\partial \Omega} K(u_n) T_k (u_n) \dvtb \\
& \qquad = \frac{1}{k}\int_{\Omega} T_k (u_n) f_n \dvta  + \frac{1}{k} \int _{\partial \Omega} T_k (u_n) g_n  \dvtb,
\end{aligned}
\end{equation*}
which yields that
\begin{equation}\label{2}
\begin{aligned}
&\frac{1}{k} \int_{\Omega} A(\nabla u_n ) \nabla T_k (u_n) \dvta  +   \frac{1}{k}\int _{\partial \Omega} K(u_n) T_k (u_n) \dvtb\\
& \qquad \leq  \frac{1}{k}\int_{\Omega} | H_n (x,  \nabla u_n) T_k (u_n) |\dvtao  +  \frac{1}{k} \int_{\Omega} |f_n| |T_k(u_n)| \dvta \\
&\qquad \quad \ \ + \frac{1}{k} \int_{\partial \Omega} |g_n| |T_k(u_n)| \dvtb
\end{aligned}
\end{equation}

Due to \eqref{73}, the sequence $T_k (u_n)$ converges to $T_k(u)$  a.e. in $\Omega$. Since $f_n$ strongly converges to $f$ in $L^1(\Omega,\vta)$ it follows that
$$
\lim _{n \rightarrow \infty} \frac{1}{k} \int_{\Omega} |f_n||T_k ( u_n)| \dvta =\frac{1}{k} \int_{\Omega}|f| |T_k(u)|\dvta .
$$

Recalling that $u$ is finite  a.e. in $\Omega$, we deduce that
$$
\lim _{k \rightarrow \infty} \lim _{n \rightarrow \infty} \frac{1}{k} \int_{\Omega}|f_n||T_k (u_n)| \dvta=0.
$$
Similarly, we obtain
$$
\lim _{k \rightarrow \infty} \lim _{n \rightarrow \infty} \frac{1}{k} \int_{\partial \Omega}|g_n||T_k (u_n)| \dvtb=0.
$$

Now, we procedure to prove 
\begin{equation}\label{9}
\lim _{k \rightarrow \infty} \limsup _{n \rightarrow \infty}  \frac{1}{k}\int_{\Omega}  |H_n (x,  \nabla u_n) T_k (u_n) |\dvtao=0.
\end{equation}
By the generalized  H\"older inequality, \eqref{21} and \eqref{8}, we have
\begin{align*}
&\frac{1}{k}\int_{\Omega}  |H_n (x,  \nabla u_n) T_k (u_n) |\dvtao\\
& \qquad\leq   \|b\|_{(\gamma , \na , 1 , \{|u_n|>\sqrt{k}\})} \||\nabla u_n|^{p-1} \|_{(\gamma , \na ^\prime , \infty , \{|u_n|>\sqrt{k}\})} \\
& \qquad \quad \ \ + \frac{\sqrt{k}}{k} \|b\|_{(\gamma , \na , 1 , \{|u_n|\leq \sqrt{k}\})} \||\nabla u_n|^{p-1} \|_{(\gamma , \na ^\prime , \infty , \{|u_n|\leq \sqrt{k} \})}\\
 & \qquad  \leq C \int ^{\vtao( \{|u_n|>\sqrt{k}\} )} _0 t^{1/ \na}  b^\ast _\gamma (t) \frac{\dt}{t} + \frac{C}{\sqrt{k}}\|b\|_{(\gamma , \na , 1)} \\
 &  \qquad \leq C \int ^{C/\ln (1+ \sqrt{k}) } _0 t^{1/ \na}  b^\ast _\gamma(t) \frac{\dt}{t} + \frac{C}{\sqrt{k}}\|b\|_{(\gamma , \na , 1)}.
\end{align*}
This proves \eqref{9}. It follows that \eqref{80} holds.\\


{\it Step 4. We prove that for any $k>0$}
\begin{equation}
\lim _{n \rightarrow \infty} \int_{\Omega} [  A ( \nabla T_k (u_n) ) -   A ( \nabla T_k(u) ) ] [\nabla T_k (u_n) - \nabla T_k(u) ] \dvta =0. \label{81}
\end{equation}\\

Let $h_\ell$ be defined by
$$
h_\ell(s)= \left\{ 
\begin{aligned}
& 0 & & \text { if }|s|>2 \ell, \\ 
&\frac{2 \ell-|s|}{\ell} & & \text { if } \ell<|s| \leq 2 \ell, \\ 
&1 & & \text { if } |s| \leq \ell .
\end{aligned}\right.
$$

Using  $h_\ell (u_n) [ T_k (u_n )-T_k(u) ]$ in \eqref{6}, we have
\begin{equation}\label{82}
\begin{aligned}
&\int_{\Omega} h_\ell (u_n)A(\nabla u_n)  (\nabla T_k (u_n ) -  \nabla T_k(u)) \dvta   \\
& \qquad = a_{k, \ell, n}+b_{k, \ell, n}+c_{k, \ell, n}+d_{k, \ell, n}+e_{k, \ell, n} + f_{k, \ell, n},
\end{aligned}
\end{equation}
with
\begin{align*}
&a_{k, \ell, n}=\int_{\Omega} h_\ell (u_n) f_n [ T_k(u_n)-T_k(u) ] \dvta, \\
&b_{k, \ell, n}=\int_{\partial \Omega} h_\ell (u_n) g_n [ T_k(u_n)-T_k(u) ] \dvtb, \\
&c_{k, \ell, n}= -\int_{\Omega} h_\ell ^{\prime} (u_n) A(\nabla u_n) \nabla u_n  [ T_k(u_n)-T_k(u) ] \dvta,\\
&d_{k, \ell, n}= -\int_{\Omega} H_n (x, \nabla u_n ) h_\ell (u_n) [ T_k(u_n)-T_k(u) ] \dvtao, \\
&e_{k, \ell, n}= -\int_{\Omega}  G_n (x , u_n) h_\ell (u_n) [ T_k(u_n)- T_k(u) ] \dvtao ,\\
&f_{k,\ell,n} =  - \int_{\partial \Omega} h_\ell (u_n)  K ( u_n ) [ T_k (u_n) -  T_k(u) ] \dvtb.
\end{align*}

We now pass to the limit in \eqref{82} first as $n\rightarrow \infty$ and then as $\ell\rightarrow \infty$.  We prove:
\begin{equation}\label{165}
\lim _{\ell \rightarrow \infty} \limsup _{n \rightarrow \infty} \int_{\Omega} h_\ell (u_n) A(\nabla u_n) [ \nabla T_k (u_n)-\nabla T_k(u) ] \dvta  =0 .
\end{equation}

Due to the point-wise convergence of $u_n$ the sequence $T_k(u_n)-T_k(u)$ converges to zero   a.e. in $\Omega$. Since $f_n \rightarrow f$ in $L^1(\Omega ; \vta)$,   $g_n \rightarrow g$ in $L^1(\partial \Omega ; \vtb)$, and $T_k(u_n) \rightarrow T_k(u)$ in $L^p(\partial \Omega ; \vtb)$, we obtain that
\begin{equation}\label{84}
\lim _{n \rightarrow \infty} a_{k, \ell , n } = \lim _{n \rightarrow \infty} b_{k, \ell , n } = 0,
\end{equation}
also
\begin{equation}\label{166}
\lim _{n \rightarrow \infty} f_{k,\ell,n} =  \int_{\partial \Omega} h_\ell (u_n)  K ( T_{2\ell}(u_n) ) [ T_k (u_n) -  T_k(u) ] \dvtb = 0 .
\end{equation}

Since
$$
|c_{k, \ell , n}| \leq \frac{2 k}{\ell} \int_{\left\{|u_n| \leq 2 \ell\right\}} A  ( \nabla u_n ) \nabla u_n \dvta
$$
and due to \eqref{80}, we obtain
\begin{equation}\label{86}
\lim _{\ell \rightarrow \infty} \limsup _{n \rightarrow \infty} c_{k, \ell , n}=0 .
\end{equation}

Now, we prove that 
\begin{equation}\label{83}
 \lim _{n\rightarrow \infty} d_{k,\ell,n}=0.
\end{equation}
Let $q_1>1$ so that $1/\na + (p-1)/p +1/q_1 =1$. From \eqref{24}, \eqref{1}, and \eqref{70},
\begin{align*}
&|d_{k,\ell,n}|=\left|\int_{\Omega} H _n (x, \nabla u_n ) h_\ell (u_n) [ T_k(u_n)-T_k(u) ] \dvtao\right| \\
& \qquad =\left|\int_{\Omega} H_n (x, \nabla T_{2 \ell} (u_n) ) h_\ell (u_n) [ T_k(u_n)-T_k(u) ] \dvtao\right| \\
& \qquad \leq \left(\int_{\Omega} b^{\na} \dvtao \right)^{1/\na} \left(\int_{\Omega}|\nabla T_{2\ell} (u_n)|^p \dvtao \right)^{(p-1)/p} \\
& \qquad \quad \ \ \cdot \left[\int_{\Omega} ( T_k(u_n)-T_k(u) ) ^{q_1} \dvtao\right]^{1/q_1}\\
& \qquad \leq (2C\ell)^{(p-1)/p} \left(\int_{\Omega} b^{\na} \dvtao \right)^{1/\na}  \left[\int_{\Omega} ( T_k(u_n)-T_k(u))^{q_1} \dvtao\right]^{1/q_1}.
\end{align*}
Then, we obtain \eqref{83}.

On the other hand, we deduce from \eqref{22} and from the fact that $0 \leq r <(N+\gamma)(p-1) /(N + \alpha -p)$,  one has
\begin{equation}\label{87}
\| |u_n|^r \|_{(\gamma , z, \infty) } \leq C .
\end{equation}

Using \eqref{25}, \eqref{87}, and generalized Hölder inequality, we get
$$
\begin{aligned}
&\|G_n(x, u_n)\|_{1 ,E , \gamma}  =\int_{ E } |G_n (x, u_n)| \\
& \qquad \leq \int_{ E } |c| |u_n|^r\dvtao  \leq \|c\|_{(\gamma , z^{\prime}, 1 , E)} \| |u_n|^r \|_{(\gamma , z, \infty , E)} \\
& \qquad \leq C \int ^{\vtao (E)} _0 c^\ast _\gamma (t) t ^{1/z^{\prime}}\frac{\dt}{t} ,
\end{aligned}
$$
which is small when $\vtao (E)$ is small. Hence $G_n(x, u_n)$ is equi-integrable. Therefore, Vitali's Theorem implies that
\begin{equation}\label{150}
G_n(x, u_n) \rightarrow G(x, u) \quad \text { in } L^1 (\Omega , \vtao).
\end{equation}
Then we conclude 
\begin{equation}\label{88}
\lim _{\ell\rightarrow \infty} \lim _{n\rightarrow \infty} e_{k,\ell,n}=0.
\end{equation}

From \eqref{84}, \eqref{86}, \eqref{83}, and \eqref{88},  we obtain that for any $k>0$
$$
\lim _{\ell \rightarrow \infty} \limsup _{n \rightarrow \infty} \int_{\Omega} h_\ell (u_n) A(\nabla u_n) [ \nabla T_k (u_n)-\nabla T_k(u) ] \dvta  =0 ,
$$
with which we prove \eqref{165}.

Recalling that for any $\ell>k$, we have
\begin{equation*}
h_\ell (u_n) A( \nabla u_n) \nabla T_k (u_n)= A ( \nabla u_n ) \nabla T_k (u_n) \quad \text {  a.e. in } \Omega .
\end{equation*}

It follows that
\begin{equation}\label{90}
\begin{aligned}
& \limsup _{n \rightarrow \infty} \int_{\Omega} A (\nabla u_n ) \nabla T_k (u_n) \dvta \\
& \qquad  \leq \lim _{\ell \rightarrow \infty} \limsup _{n \rightarrow \infty} \int_{\Omega} h_\ell (u_n) A( \nabla u_n ) \nabla T_k(u) \dvta .
\end{aligned}
\end{equation}

According to the definition of $h_\ell$, we have
$$
h_\ell (u_n) A ( \nabla u_n )=h_\ell (u_n) A ( \nabla T_{2 \ell} (u_n) ) \quad \text {  a.e. in } \Omega,
$$
so that \eqref{73} and \eqref{72} give
\begin{equation}\label{91}
\lim _{n \rightarrow \infty} \int_{\Omega} h_\ell (u_n) A ( \nabla u_n ) \nabla T_k (u) \dvta=\int_{\Omega} h_\ell (u) V_{2 \ell} \nabla T_k(u) \dvta .
\end{equation}

If $\ell>k$, we have
$$
A ( \nabla T_\ell (u_n) ) \chi_{\{|u_n|<k\}}=  A(  \nabla T_k (u_n ) ) \chi_{\{|u_n|<k\}}
$$
a.e. in $\Omega$. From \eqref{73} and \eqref{72}, it follows that
$$
V_\ell \chi_{\{|u|<k\}}= V_k \chi_{\{|u|<k\}} \quad  \text {  a.e. in } \Omega \backslash\{|u|=k\},
$$
and then we obtain for any $\ell>k$
$$
V_\ell \nabla T_k(u)= V_k \nabla T_k(u) \quad  \text {  a.e. in } \Omega.
$$

Therefore, \eqref{90} and \eqref{91}  allow us to conclude
\begin{equation}\label{92}
\limsup _{n \rightarrow \infty} \int_{\Omega} A (\nabla T_k (u_n )) \nabla T_k  (u_n ) \dvta \leq \int_{\Omega} V_k \nabla T_k(u) \dvta .
\end{equation}

Now, we are in a position to prove \eqref{81}. Indeed, the monotone character of $A$ implies that for any $n$,
\begin{equation}\label{93}
0 \leq \int_{\Omega} [ A( \nabla T_k (u_n)) - A(  \nabla T_k (u )  ) ] [ \nabla T_k (u_n)-\nabla T_k(u)] \dvta.
\end{equation}

 Writing
$$
\begin{aligned}
& \int_{\Omega} [ A(\nabla T_k(u_n))-A (\nabla T_k(u))]   [\nabla T_k(u_n)-\nabla T_k(u)] \dvta \\
&\qquad =  \int_{\Omega} A( \nabla T_k (u_n)) [ \nabla T_k(u_n)-\nabla T_k(u) ]  \dvta \\
& \qquad \quad \ \ -\int_{\Omega} A( \nabla T_k(u)) [\nabla T_k(u_n)-\nabla T_k(u)] \dvta.
\end{aligned}
$$
Using \eqref{92} and \eqref{93} allows one to conclude that \eqref{81} holds for any $k>0$.\\


{\it Step 5. We prove that for any $k>0$}
\begin{gather}
A(\nabla T_k(u))=V_k, \label{13} \\
A(\nabla T_k (u_n)) \nabla T_k(u_n) \rightharpoonup A (\nabla T_k(u)) \nabla T_k(u) \quad \text { {\it in} } L^1(\Omega ; \vta). \label{14}
\end{gather}\\

From \eqref{81}, we have for any $k>0$
\begin{equation}\label{15}
\lim _{n \rightarrow \infty} \int_{\Omega} A( \nabla T_k (u_n)) \nabla T_k (u_n) \dvta = \int_{\Omega} V_k \nabla T_k(u) \dvta.
\end{equation}

The classical arguments, known as Minty arguments (see \cite{leray1965quelquesresulatat, lions1969quelquesmethodes}), allow us to identify  $V_k$ with $A( \nabla T_k(u))$. Let $\phi \in (L^{\infty}(\Omega))^N$. By \eqref{72} and \eqref{15}, it follows that for any $t \in \mathbb{R}$
\begin{align*}
& \lim _{n \rightarrow \infty} \int_{\Omega} [  A( \nabla T_k( u_n ) )- A ( \nabla T_k (u) +t \phi ) ]    [\nabla T_k( u_n) - \nabla T_k (u) - t \phi ] \dvta \\
& \qquad = - \int_{\Omega} [ V_k - A ( \nabla T_k (u ) +t  \phi ) ]  t \phi \dvta.
\end{align*}

Using the monotone character \ref{136} of $A$, we obtain that for any $t \neq 0$
$$
-\operatorname{sign}(t) \int_{\Omega}  [ V_k -A( \nabla T_k ( u ) +t \phi) ] \phi \dvta \geq 0 .
$$

Since $A( \nabla T_k (u)+t \phi)$ converges strongly to $A ( \nabla T_k (u ) )$ in $(L^{p^{\prime}}(\Omega ; \vta ) )^N$ as $t$ goes to zero, letting $t \rightarrow 0$ in the above inequality leads to
$$
\int_{\Omega} [ V_k - T_k (\nabla u) ] \phi\dvta =0.
$$
for any $\phi \in (L^{\infty}(\Omega))^N$. We  conclude that \eqref{13}.

From \eqref{81}, we get
$$
[ A( \nabla T_k (u_n ) )-A( \nabla T_k(u))]  [\nabla T_k (u_n)-\nabla T_k(u)] \rightarrow 0, \quad \text { in } L^1(\Omega , \vta).
$$
Using \eqref{71}   leads to \eqref{14}.\\


{\it Step 6.  We claim that}
\begin{equation}\label{102}
 \nabla u_n  |_{\Omega _m } \text { {\it is a Cauchy sequence in measure}}.
\end{equation}\\

Here, we follow an argument used in the proof of \cite[Lemma 1]{boccardo1992nonlinearhandsidemeasure}. To prove \eqref{102}, given $\varepsilon>0$ and $\varepsilon _1>0$ fixed, we set, for some $L>1$ and $\delta \in (0,1)$:
\begin{align*}
&E_1= \{ |\nabla T_k( u_{n_1})|>L\} \cup   \{ |\nabla T_k( u_{n_2} )|> L \} \cup \{ | u_{n_1}  |> L\} \cup \{| u_{n_2} |>L\}, \\
&E_2= \{| T_k (u_{n_1})  - T_k( u_{n_2} ) |>\delta ^2 \} ,\\
&E_3=\{|T_k( u_{n_1} ) - T_k( u_{n_2}) | \leq \delta ^2, \ |\nabla T_k( u_{n_1} )| \leq L, \ |\nabla T_k( u_{n_2} ) | \leq L, \\
& \qquad | u_{n_1}  | \leq L ,  \ | u_{n_2} | \leq L, \ |\nabla  T_k(u_{n_1}) - \nabla T_k( u_{n_2}) | > \varepsilon\} .
\end{align*}
Observe that 
$$
\{|\nabla T_k( u_{n_1} )  -  \nabla  T_k( u_{n_2})|> \varepsilon\}  \subset E_1 \cup E_2 \cup E_3.
$$

From \eqref{96}, $(T_k (u_n))$ and $(|\nabla T_k (u_n)|)$ are bounded in $L^1(\Omega _m ; \vta)$; we have  
$$
\vta (E_1 \cap \Omega _m) < \varepsilon _1,
$$
for $L>k$ large enough, independently of $n_1$ and  $n_2$. 

Now, let's consider $\vta (E_3 \cap \Omega _m)$. Define
$$
K_0=\{( \xi, \eta) \in \mathbb{R}^{2N} \:|\: |\xi| \leq L , \ |\eta| \leq L, \ |\xi-\eta| \geq \varepsilon\},
$$
then
$$
\inf \{[ A (\xi)- A(\eta)] [\xi-\eta] \:|\:(\xi, \eta) \in K_0\}=C>0,
$$
since $K_0$ is compact.

Hence,
\begin{align*}
&C \vta (E_3 \cap \Omega _m)  \\
& \qquad \leq \int_{E_3 \cap \Omega _m} [ A ( \nabla T_k ( u_{n_1}) )-A ( \nabla T_k ( u_{n_2}) ) ] [ \nabla T_k (u_{n_1} ) - \nabla T_k ( u_{n_2}) ] \dvta \\
& \qquad \leq \int_{\Omega }  h_{L}(u_{n_1})h_{L}(u_{n_2}) h_{\delta} ( T_k(u_{n_1}) - T_k(u_{n_2}))   \\
& \qquad \quad \ \ \cdot [ A ( \nabla T_k ( u_{n_1}) )-A ( \nabla T_k ( u_{n_2}) ) ] [ \nabla T_k (u_{n_1} ) - \nabla T_k ( u_{n_2}) ] \dvta.
\end{align*}
Substituting  $h_{L}(u_{n_1})h_{L}(u_{n_2}) h_{\delta} (T_k(u_{n_1}) -  T_k(u_{n_2}))   [T_k(u_{n_1})-T_k(u_{n_2})]$  into \eqref{6} and considering also
\begin{align*}
&\int _{\Omega} \{ h_{L}^\prime (u_{n_1})h_{L}(u_{n_2}) h_{\delta} (T_k(u_{n_1} )- T_k(u_{n_2})) A(\nabla u_{n_1}) \nabla u_{n_1}  \\
&\qquad \quad  \ \ +  h_{L}(u_{n_1}) h_{L} ^{\prime}(u_{n_2})  h_{\delta} ( T _k (u_{n_1}) - T_k (u_{n_2})  ) A(\nabla u_{n_1})\nabla u_{n_2} \\
&\qquad \quad \ \ + h_{L}(u_{n_1}) h_{L}(u_{n_2}) h_{\delta} ^\prime (T_k(u_{n_1} )- T_k(u_{n_2}))\\
&\qquad \quad \ \ \cdot A(\nabla u_{n_1})[\nabla T_k(u_{n_1}) - \nabla T_k(u_{n_2}) ] \} (T_k(u_{n_1})-T_k(u_{n_2}))\dvta\\
& \qquad \leq C L\delta,
\end{align*}
because \ref{136} and \eqref{7}. Proceeding as in Step 4, for a suitable constant $\delta \in (0,1)$, we have that there is $N_0 >0$ such that
$$
 \vta (E_3 \cap \Omega _m) < \varepsilon _1, \quad \text { if } n_1, n_2 > N_0.
$$

By \eqref{97}, there exists $N_1>0$ such that
$$
\vta (E_2 \cap \Omega _m) < \varepsilon _1, \quad \text { if } n_1, n_2 > N_1. 
$$

Hence,
\begin{equation}\label{103}
 \nabla T _k (u_n)  |_{\Omega _m } \text { {\it is a Cauchy sequence in measure}}.
\end{equation}

Finally, if $\varepsilon >0$ is fixed, we have
\begin{align*}
&\{|\nabla u_{n_1} - \nabla u_{n_2}| > \varepsilon\} \\
& \qquad \subset \{|u_{n_1}| >L_0 \} \cup \{|u_{n_2}| > L_0 \} \cup \{|\nabla T_{L_0} ( u_{n_1} ) - \nabla T_{L_0} ( u_{n_2})| >\varepsilon\}.
\end{align*}
Therefore, \eqref{8} and \eqref{103} imply \eqref{102}.

On the other hand, from a computation similar to \eqref{50},  using \eqref{102}, we see
\begin{equation}\label{151}
H_n (x , \nabla u_n) \rightarrow H (x , \nabla u) \quad \text { in } L^1 (\Omega ; \vta).
\end{equation}\\


{\it Step 7. We now proceed to take the limit in the approximated problem.}\\

Let $h \in W^{1, \infty}(\mathbb{R})$ with compact support  contained in the interval $[-k, k]$, where $k>0$, and let $\varphi \in W^{1, p}(\Omega ; \vta) \cap L^{\infty}(\Omega)$. Using $ h (u_n) \varphi$ as a test function in the approximated problem, we have
\begin{equation}\label{101}
\begin{aligned}
&\int_{\Omega} h (u_n) A(\nabla u_n)  \nabla \varphi \dvta + \int_{\Omega}  h ^\prime (u_n) \varphi A(\nabla u_n)  \nabla u_n \dvta \\
& \qquad  + \int_{\Omega}  H_n (x,  \nabla u_n) h (u_n) \varphi \dvtao + \int_{\Omega} G_n(x, u_n) h (u_n) \varphi \dvtao \\
& \qquad + \int _{\partial \Omega} K(u_n) h (u_n) \varphi \dvtb 
=\int_{\Omega} h (u_n) \varphi f_n \dvta  + \int _{\partial \Omega} h (u_n) \varphi g_n  \dvtb.
\end{aligned}
\end{equation}

We want to pass to the limit in this equality. Since $\operatorname{supp} (h)$ is contained in the interval $[-k, k]$, and considering that  $f_n \rightarrow f$ in $L^1(\Omega ; \vta )$,  $g_n \rightarrow g$ in $L^1 (\partial \Omega ; \vtb)$, $T_k (u_n) \rightarrow T_k(u)$ in $L^p (\partial \Omega ; \vtb)$, and by \eqref{73}, we  obtain
\begin{gather*}
\lim _{n \rightarrow \infty} \int_{\Omega} f_n  h(u_n)\varphi \dvta =\int_{\Omega} f  h(u) \varphi \dvta ,\\
\lim _{n \rightarrow \infty} \int_{\partial \Omega} g_n  h(u_n)\varphi \dvtb =\int_{\partial \Omega} g  h(u) \varphi \dvtb,\\
\lim _{n \rightarrow \infty} \int _{\partial \Omega} K(u_n) h (u_n) \varphi \dvtb  =\int _{\partial \Omega} K(u) h (u) \varphi \dvtb.
\end{gather*}

In view of \eqref{72} and \eqref{13},
$$
\lim _{n \rightarrow \infty} \int_{\Omega} h(u_n) A( \nabla T_k (u_n )) \nabla \varphi \dvta  =\int_{\Omega} h(u) A( \nabla T_k(u)) \nabla \varphi \dvta.
$$

From \eqref{14}, we get
$$
\lim _{n \rightarrow \infty} \int_{\Omega} h^{\prime} (u_n) \varphi A( \nabla T_k (u_n )) \nabla T_k (u_n)  \dvta =\int_{\Omega} h^{\prime}(u) \varphi A( \nabla T_k(u)) \nabla T_k(u)  \dvta .
$$

By \eqref{150} and \eqref{151}, up to a subsequence still indexed by $n$,
\begin{gather*}
\lim _{n \rightarrow \infty}  \int_{\Omega}  G_n (x, u_n) h (u_n) \varphi \dvtao = \int_{\Omega}  G (x,  u) h (u) \varphi \dvtao.\\
\lim _{n \rightarrow \infty}  \int_{\Omega}  H_n (x,  \nabla u_n) h (u_n) \varphi \dvtao = \int_{\Omega}  H_n (x,  \nabla u) h (u) \varphi \dvtao.
\end{gather*}

Therefore, by passing to the limit in \eqref{101}, we obtain condition \eqref{104} in the definition of a renormalized solution. The decay of the truncated energy \eqref{105} is a consequence of \eqref{80} and \eqref{14}.

Additionally, from \eqref{70}, $(\nabla u_n)$ is equi-integrable, and by $\nabla T_k(u_n) \rightarrow \nabla T_k (u)$ a.e. in $\Omega$, we have
$$
\nabla T_k(u_n) \rightarrow \nabla T_k (u) \quad \text { in } L^1(\Omega ; \vta).
$$ 

Moreover, by Fatou's lemma, we deduce from \eqref{70} that
$$
\int _{\Omega} |\nabla T_k (u)|^p \dvta + \int _{\Omega} |\nabla T_k (u)|^p \dvta \leq Ck.
$$
Then, Lemma \ref{42} implies $|u|^{p-1} \in L^{\past / p, \infty}(\Omega; \vtao)$, $|u|^{p-1} \in L^{q/p , \infty} (\partial  \Omega ; \vtb)$, and $|\nabla u|^{p-1} \in L^{\na^{\prime}, \infty}(\Omega; \vtao)$. Recall that, from Step 2, $u$ is finite a.e. in $\Omega$, and $T_k(u) \in W^{1, p}(\Omega;\vta)$ for any $k>0$. We can conclude that $u$ is a renormalized solution to \eqref{5}. \end{proof}

\subsection{Stability result}\label{167}


For any $n\in \mathbb{N}$, let $u_n$ be a renormalized solution  to the problem
\begin{equation}  \label{142}
\left\{ 
\begin{aligned}
-\operatorname{div}(\vta A( \nabla u) )+ \vtao H_n(x,\nabla u) + \vtao G_n(x,u)&= f _n \vta & & \text { in } \Omega, \\ 
 \vta A(\nabla u) \cdot  \nu + \vtb K(u) &= g_n \vtb & &  \text { on } \partial \Omega,
\end{aligned}
\right.
\end{equation}
where $f_{n}  \in L^1(\Omega , \vta )$, $g_{n} \in L^1(\partial \Omega ; \vtb)$, and
\begin{equation}\label{143}
 f_{n} \rightarrow f \text {  in } L^1(\Omega , \vta ) \quad \text { and } \quad  g_{n} \rightarrow g \text {  in } L^1(\partial \Omega ; \vtb).
\end{equation}

Also, $H_n : \Omega \times \mathbb{R}^N \rightarrow \mathbb{R}$, $G_n : \Omega \times \mathbb{R} \rightarrow \mathbb{R}$  are Carathéodory functions satisfying:
\begin{gather}
| H_n (x, \xi) |\leq b(x)|\xi|^{p-1}, \quad b\in L^{\na , 1}(\Omega ; \vtao),\\
G_n(x, s) s \geq 0, \quad |G_n(x,s)|\leq c(x) |s|^{r}, \quad c \in L^{z^{\prime}, 1}(\Omega ; \vtao), 
\end{gather}
for almost every $x \in \Omega$ and for every $s \in \mathbb{R}$ and $\xi \in \mathbb{R}^N$, where $\na$, $r$, $z$ verify  \eqref{78}. Furthermore, for almost every $x\in \Omega$
\begin{equation} \label{153}
H_n(x , \xi _n)\rightarrow H(x,\xi ) \quad \text { and } \quad G_n(x , s _n)\rightarrow G(x,s ), 
\end{equation}
for every sequences $(\xi _n) \subset \mathbb{R}^N$ and $( s_n) \subset \mathbb{R}$, where $H$ and $G$ are Carathéodory functions.

\renewcommand*{\proofname}{{\bf Proof of Theorem \ref{155}}}
\begin{proof}
We mainly follow the arguments developed in the proof of Theorem \ref{146}. 

Using $h=h_\ell$, defined by
$$
h_\ell(s)= \left\{ 
\begin{aligned}
& 0 & & \text { if }|s|>2 \ell, \\ 
&\frac{2 \ell-|s|}{\ell} & & \text { if } \ell<|s| \leq 2 \ell, \\ 
&1 & & \text { if } |s| \leq \ell ,
\end{aligned}\right.
$$
and $\varphi=T_k (u_n)$ in the renormalized formulation \eqref{104} we have, for any $\ell>0$ and any $k>0$
\begin{equation}\label{144}
\begin{aligned}
&\int_{\Omega} h_\ell (u_n) A(\nabla u_n)  \nabla T_k (u_n) \dvta + \int_{\Omega} h^{\prime} _\ell (u_n) A(\nabla u _n) \nabla u_n T_k (u_n) \dvta  \\
&\qquad \quad \ \ + \int_{\Omega} H _n (x,\nabla u _n) T_k (u_n) h_\ell (u_n) \dvtao  + \int_{\Omega} G _n(x, u_n) T_k (u_n) h_\ell (u_n) \dvtao \\
& \qquad \quad \ \ + \int_{\partial \Omega} K(u_n) T_k (u_n) h_\ell (u_n) \dvtb \\
& \qquad = \int_{\Omega}  T_k (u_n) h_\ell(u_n) f_n \dvta + \int_{\partial \Omega}  T_k (u_n) h_\ell(u_n) g_n \dvtb.
\end{aligned}
\end{equation}

We now pass to the limit as $\ell \rightarrow \infty$: 

In view of the definition of $h_\ell$ for any $\ell >k$ we have
$$
\int_{\Omega} h_\ell (u_n) A(\nabla u_n)  \nabla T_k (u_n) \dvta =\int_{\Omega} A(\nabla u_n) \nabla T_k (u_n) \dvta.
$$

Due to \eqref{105}, we get
$$
\lim _{\ell \rightarrow \infty} \int_{\Omega} h_\ell ^{\prime} (u_n) A( \nabla u_n) \nabla u_n T_k (u_n ) \dvta=0 .
$$

It follows that passing to the limit as $\ell \rightarrow \infty$  in  \eqref{144} leads to
\begin{equation}\label{145}
\begin{aligned}
&\int_{\Omega}  A(\nabla u_n)  \nabla T_k (u_n) \dvta   + \int_{\Omega} H_n (x,\nabla u _n) T_k (u_n)  \dvtao  + \int_{\Omega} G_n (x, u_n) T_k (u_n)  \dvtao \\
&\qquad + \int_{\partial \Omega} K(u_n) T_k (u_n)  \dvtb = \int_{\Omega} T_k (u_n)  f_n \dvta + \int_{\partial \Omega}  T_k (u_n) g_n \dvtb.
\end{aligned}
\end{equation}
Hence, proceeding as in step 1 of the proof Theorem \ref{146}, we have
\begin{gather*}
T_k(u_n ) \text { is bounded in } W^{1, p}(\Omega ; \vta), \\
A( \nabla T_k (u_n ) ) \text { is  bounded in }  (L^{p^{\prime}}(\Omega ; \vta) )^N \text { for any } k>0.
\end{gather*}

Using a similar process to one used to obtain \eqref{145} we get
\begin{equation*}
\begin{aligned}
&\int_{\Omega} A(\nabla u_n)  \nabla \Psi_p (u_n)  \dvta +\int_{\Omega}  H_n (x,  \nabla u_n) \Psi_p (u_n) \dvtao \\
&\qquad \quad \ \ + \int_{\Omega} G_n(x, u_n) \Psi_p (u_n) \dvtao  + \int _{\partial \Omega} K(u_n) \Psi_p (u_n) \dvtb\\
&\qquad =\int_{\Omega} \Psi_p (u_n) f_n \dvta  + \int _{\partial \Omega} \Psi_p (u_n) g_n  \dvtb, 
\end{aligned}
\end{equation*}
where 
$$
\Psi_p(r)=\int_0^r \frac{1}{(1+|t|)^p} \dt, \quad \forall r \in \mathbb{R}.
$$

By the argument in step 2 of the proof of Theorem \ref{146} we obtain
\begin{equation*}
\begin{gathered}
 \vta (\{x \in \Omega \:|\: |u_n(x)| > L \}) \leq \frac{C}{\ln (1+L)} ,\\
\vtb (\{x \in \partial \Omega \:|\: |u_n(x)| > L \}) \leq \frac{C}{\ln (1+L)},
 \end{gathered}
\end{equation*}
for all $n$, where $C>0$ is a constant independent of $n$.

Furthermore, there exist a measurable functions $u: \Omega \rightarrow \bar{\mathbb{R}}$, $w: \partial \Omega \rightarrow \bar{\mathbb{R}}$, and a field $V_k \in (L^{p^{\prime}}(\Omega ; \vta))^N$ such that, up to a subsequence still indexed by $n$,
\begin{gather*}
T_k (u_n )|_{\Omega _m}, \ u_n |_{\Omega _m}  \text { and } \textnormal{tr} (u_n )  \text { are Cauchy sequences in measure},\\
u_n \rightarrow u  \text {  a.e. in } \Omega ,  \text { and }  u_n \rightarrow w  \text {  a.e. on } \partial \Omega, \\
T_k (u_{n}) \rightharpoonup T_k(u) \quad  \text {  in } W^{1, p}(\Omega ; \vta), \\
\tau (T_k (u)) = T_k (w),\\
A(\nabla T_k(u_n)) \rightharpoonup V _k \quad \text { in } (L^{p^{\prime}}(\Omega ; \vta ))^N, 
\end{gather*}
for all $k>0$. Also, since $W^{1, p} (\Omega ; \vartheta_\alpha ) \hookrightarrow L^p (\partial \Omega ; \vartheta_\beta)$, we can assume that
\begin{equation*}
T_k (u_{n}) \rightarrow T_k(u) \quad \text {  in } L^{ p}(\partial \Omega ; \vtb).
\end{equation*}

Following the arguments developed in steps 3-5 of the proof of Theorem \ref{146}, we can conclude \eqref{148}. Furthermore, by repeating the same arguments, we can demonstrate that $u$ is a renormalized solution to \eqref{5}.\end{proof}

\appendix

\section{ }\label{139}

In this appendix, we prove the existence of weak solutions to the problem \eqref{7}. We study the existence of solutions for the problem
\begin{equation}\label{141}
\left\{\begin{aligned}
-\operatorname{div}\left( \vta A(\nabla u)\right)+\vtao H_n(x,  \nabla u)+ \vta G_n (x, u) &= f _n & & \text { in } \Omega_n, \\
 \vta A(\nabla u)\cdot \nu + \vtb K(u)&= g_n   & & \text { on } \partial \Omega _n \cap \partial \Omega,
\end{aligned}\right.
\end{equation}
where $\Omega _n = \Omega \cap \{|x|<n\}$ and $n\in \mathbb{N}$. Recall that $f_n = 0$ on $\Omega \backslash \Omega _n$, and $g_n = 0$ on  $\partial \Omega \backslash \partial \Omega _n$.

Define
$$
W|_{\mathcal{D}} = \{ u \in W^{1,p} (\Omega _n; \vta) \:|\: u|_{\partial \Omega _n \backslash \partial \Omega} = 0\}.
$$

Let $w\in W|_{\mathcal{D}} $. According to the Minty-Browder theorem, there exists a unique  $u\in W|_{\mathcal{D}} $ such that  
\begin{equation}\label{106}
\begin{aligned}
&\int_{\Omega _n} A(\nabla u)  \nabla v \dvta +\int_{\Omega _n}  H_n (x,  \nabla w) v \dvtao + \int_{\Omega _n} G_n(x, w) v \dvtao \\
&\qquad + \int _{\partial \Omega _n \cap \partial \Omega} K(w) v \dvtb 
 =\int_{\Omega _n } v f_n \dvta  + \int _{\partial \Omega _n \cap \partial \Omega} v g_n  \dvtb,
\end{aligned}
\end{equation}
for all $v\in W|_{\mathcal{D}}$.  

It follows that we can consider the functional $T :  W|_{\mathcal{D}} \rightarrow  W|_{\mathcal{D}} $ defined by
$$
T (w) = u, \quad \forall w \in  W|_{\mathcal{D}},
$$
where $u$ is the unique element of $ W|_{\mathcal{D}}$ verifying \eqref{106}. We now prove that $T$ is a continuous and compact operator.\\

{\it Step 1. $T$ is continuous.}\\

 Let $w_m \in W|_{\mathcal{D}}$ such that $w_m \rightarrow w$ in $W|_{\mathcal{D}}$. Up to a subsequence (still denoted by $w_n$), we have
\begin{gather}
w_m \rightarrow w \quad \text {  in } L^p(\partial \Omega _n  \cap \partial \Omega ; \vtb ),  \label{117}\\
w_m \rightarrow w \quad \text { a.e. in } \Omega _n, \label{118}\\
w_m \rightarrow w \quad \text { a.e. in } \partial \Omega _n  \cap \partial \Omega , \label{119}\\
 \nabla w_m \rightarrow  \nabla w \quad \text {  a.e. in } (\Omega _n )^N. \label{120}
\end{gather}

Write $u_m= T ( w_m )$. With $w_m$ in place of $w$ and choosing $u_m$ as a test function in \eqref{106}, from \ref{136}, \eqref{24}, and $\eqref{25}$, we obtain that
$$
\sigma \int_{\Omega _n} |\nabla u_m |^p \dvta \leq \int_{\Omega _n} n|u_m| \dvtao  + \int_{\Omega _n} | u_m f_n | \dvta + \int_{\partial \Omega _n \cap \partial \Omega} | u_m g_n| \dvtb.
$$

From the Poincaré inequality and Young's inequality, we get
\begin{equation}\label{140}
\int_{\Omega _n} |\nabla u_m |^p \dvta +  \int_{\Omega _n}| u_m |^p \dvta \leq C
\end{equation}
where $C>0$ is a constant independent of $m$. 

Consequently, there exists a subsequence (still denoted by $u_m$), a measurable function $u$, and a field $V$  such that
\begin{gather}
u_m \rightharpoonup u \quad \text {  in } W^{1, p}(\Omega _n ; \vta), \label{107}\\
u_m \rightarrow u \quad \text {  in } L^p(\Omega _n ; \vta ), \label{111}\\
u_m \rightarrow u \quad \text {  in } L^p(\partial \Omega _n  \cap \partial \Omega ; \vtb ), \label{108}\\
u_m \rightarrow u \quad \text { a.e. in } \Omega _n, \label{109}\\
u_m \rightarrow u \quad \text { a.e. in } \partial \Omega _n  \cap \partial \Omega , \label{112}\\
A( \nabla u_m ) \rightharpoonup  V \quad \text { in } (L^{p^{\prime}}(\Omega _n; \vta) )^N . \label{110}
\end{gather}

To prove the continuity of $T$, it remains to show that $u = T(w)$, i.e., $u$ satisfies \eqref{106}. Using \eqref{106} with $w_m$ in place of $w$ and the test function $u_m - u$, we have
\begin{equation}\label{113}
\begin{aligned}
&\int_{\Omega _n} A(\nabla u_m) [ \nabla u_m - \nabla u] \dvta +\int_{\Omega _n}  H_n (x,  \nabla w_m) (u_m-u) \dvtao \\
& \qquad \quad \ \ + \int_{\Omega _n} G_n (x, w _m ) (u_m-u) \dvtao  + \int _{\partial \Omega _n \cap \partial \Omega} K(w_m) (u_m-u) \dvtb \\
&\qquad =\int_{\Omega _n } (u_m-u) f_n \dvta  + \int _{\partial \Omega _n \cap \partial \Omega} (u_m-u) g_n  \dvtb.
\end{aligned}
\end{equation}

Then, from \eqref{24}, \eqref{25}, \eqref{119}, and \eqref{111} - \eqref{112},
\begin{equation}\label{114}
\lim _{m \rightarrow \infty} \int_{\Omega _n} A(\nabla u_m )[ \nabla u_m - \nabla u] \dvta=0 .
\end{equation}

Now, we use the Minty arguments to identify $V$ with $A(\nabla u)$. Let $\phi \in (L^{\infty}(\Omega _n))^N$. By \eqref{110} and \eqref{114}, it follows that for any $t \in \mathbb{R}$
\begin{align*}
& \lim _{m \rightarrow \infty} \int _{\Omega _n} [ A( \nabla u_m )- A(\nabla u  + t \phi )][\nabla u_m - \nabla u -t \phi ] \dvta \\
& \qquad  = - \int_{\Omega _n} [V - A( \nabla u+t \phi)] t \phi \dvta.
\end{align*}

Using the monotone character $A$, we obtain that for any $t \neq 0$,
$$
-\operatorname{sign}(t) \int_{\Omega _n} [V - A( \nabla u+t \phi)] \phi \dvta \geq 0 .
$$

Since $A(\nabla u+t \phi)$ converges strongly to $A( \nabla u)$ in $(L^{p^{\prime}}(\Omega _n ; \vta))^N$ as $t$ goes to zero, letting $t \rightarrow 0$ in the above inequality leads to
$$
\int_{\Omega _n}[V - A( \nabla u)] \phi \dvta=0,
$$
for any $\phi \in (L^{\infty}(\Omega _n))^N$. We  conclude that
\begin{equation}\label{115}
V = A( \nabla u).
\end{equation}

By using \eqref{118} - \eqref{120}, \eqref{110}, and \eqref{115}, we can pass to the limit as $n \rightarrow \infty$ in \eqref{106} with $w_n$ in place of $w$, and we get
\begin{equation*}
\begin{aligned}
&\int_{\Omega _n} A(\nabla u)  \nabla v \dvta +\int_{\Omega _n}  H_n (x,  \nabla w) v \dvtao + \int_{\Omega _n} G_n(x, w) v \dvtao \\
& \qquad + \int _{\partial \Omega _n \cap \partial \Omega} K(w) v \dvtb  =\int_{\Omega _n } v f_n \dvta  + \int _{\partial \Omega _n \cap \partial \Omega} v g_n  \dvtb,
\end{aligned}
\end{equation*}
for all $v\in W|_{\mathcal{D}}$. It follows that $T$ is continuous.\\

{\it Step 2. $T$ is  compact.} \\

Let $(w_m)\subset W|_{\mathcal{D}}$ be a bounded sequence. Up to a subsequence (still denoted by $w_n$), we have
\begin{gather}
w_m \rightarrow w \quad \text {  in } L^p( \Omega _n   ; \vta ), \label{121}\\
w_m \rightarrow w \quad \text {  in } L^p(\partial \Omega _n  \cap \partial \Omega ; \vtb ),  \label{122}\\
w_m \rightarrow w \quad \text { a.e. in } \Omega _n, \label{123}\\
w_m \rightarrow w \quad \text { a.e. in } \partial \Omega _n  \cap \partial \Omega , \label{124}\\
 \nabla w_m  \rightharpoonup  \nabla w \quad \text {  a.e. in } (\Omega _n )^N. \label{125}
\end{gather}

Write $u_m = T(w_m)$. We have
\begin{equation}\label{128}
\begin{aligned}
&\int_{\Omega _n} A(\nabla u_m)  \nabla v \dvta +\int_{\Omega _n}  H_n (x,  \nabla w_m) v \dvtao + \int_{\Omega _n} G_n(x, w_m) v \dvtao \\
& \qquad+ \int _{\partial \Omega _n \cap \partial \Omega} K(w_m) v \dvtb  =\int_{\Omega _n } v f_n \dvta  + \int _{\partial \Omega _n \cap \partial \Omega} v g_n  \dvtb,
\end{aligned}
\end{equation}
for all $v\in W|_{\mathcal{D}}$. 

Substituting $v=u_m$ into \eqref{128}, similarly as in the firs step, we obtain
\begin{equation}\label{129}
\int_{\Omega _n} |\nabla u_m |^p \dvta +  \int_{\Omega _n}| u_m |^p \dvta \leq C,
\end{equation}
where $C>0$ is a constant independent of $m$. 

Hence, there exists a subsequence (still denoted by $u_m$), a measurable function $u$, and a field $V$  such that
\begin{gather}
u_m \rightharpoonup u \quad \text {  in } W^{1, p}(\Omega _n ; \vta), \label{130}\\
u_m \rightarrow u \quad \text {  in } L^p(\Omega _n ; \vta ), \label{131}\\
u_m \rightarrow u \quad \text {  in } L^p(\partial \Omega _n  \cap \partial \Omega ; \vtb ), \label{132}\\
u_m \rightarrow u \quad \text { a.e. in } \Omega _n, \label{133}\\
u_m \rightarrow u \quad \text { a.e. in } \partial \Omega _n  \cap \partial \Omega . \label{134}
\end{gather}

Next, we prove 
\begin{equation}\label{126}
 \nabla u_m  \rightarrow \nabla u \quad \text {  in } (L^p( \Omega _n   ; \vta ))^N.
\end{equation}

To establish \eqref{126}, first, we prove 
\begin{equation} \label{127}
\nabla u_m   \text { {\it is a Cauchy sequence in measure}}.
\end{equation}
 Let $\varepsilon >0$ and $\varepsilon _1 >0$ be fixed. We set, for some  $L>0$ and $\delta>0$:
\begin{align*}
&E_1= \{|\nabla u_{m_1} |>L\}  \cup  \{|\nabla u_{m_2}|> L\}, \\
&E_2=\{ \ |\nabla u_{m_1}| \leq L, \ |\nabla u_{m_2}| \leq L, \ |\nabla u_{m_1} - \nabla u_{m_2}| > \varepsilon\}.
\end{align*}

Observe that 
$$
\{|\nabla u_{m_1}   -  \nabla u_{m_2}|> \varepsilon\}  \subset E_1 \cup E_2 .
$$

Since  $(u_m)$ and $(|\nabla u_m |)$ are bounded in $L^p(\Omega _n ; \vta)$, one has  
$$
\vta (E_1) < \varepsilon _1,
$$
for $L$ large enough, independently of $m_1$ and  $m_2$. 

Now, consider  $\vta (E_2)$. Define
$$
K_0=\{( \xi, \eta) \in \mathbb{R}^{2N} \:|\: |\xi| \leq L , \ |\eta| \leq L, \ |\xi-\eta| \geq\varepsilon\},
$$
then
$$
\inf \{[ A (\xi)- A(\eta)] [\xi-\eta] \:|\:(\xi, \eta) \in K_0\}=C>0,
$$
since $K_0$ is compact.

Hence,
\begin{equation} \label{135}
C \vta (E_2)  \leq \int_{E_2} [ A ( \nabla u_{m_1} ) - A ( \nabla  u_{m_2} ) ] [ \nabla u_{m_1}  - \nabla  u_{m_2} ] \dvta .
\end{equation}
Substituting  $u_{m_1}- u_{m_2}$ into \eqref{128}, for a suitable constant, we have there is $N_0 >0$ such that
$$
\vta (E_2) < \varepsilon _1, \quad \text { if } m_1, m_2 >N_0.
$$

Therefore, we conclude \eqref{127}.

Finally, substituting $v=u_m$ into \eqref{128}, we obtain that $(|\nabla u_m|^p)$ is uniformly integrable. Then, by Vitali's Convergence Theorem,  which proves \eqref{126}. It follows that $T$ is compact.

To finalize this appendix, proceeding similarly as in \eqref{140}, we have
$$
\int_{\Omega _n }|T(w)|^p \dvta \leq C, \quad \forall w\in W|_{\mathcal{D}},
$$
where $C$ is a constant depending on $\Omega_n$, $n$, $\alpha$, $f$, and $g$. Then, using Schauder's fixed-point theorem ensures the existence of at least one fixed point $T(w)=w$, proving the existence of a weak solution to the problem \eqref{141}.


\vspace{1cm}

\noindent {\bf Author contributions:} All authors have contributed equally to this work for writing, review and editing. All authors have read and agreed to the published version of the manuscript.

\noindent {\bf Funding:} This work were supported by National Institute of Science and Technology of Mathematics INCT-Mat, CNPq, Grants 170245/2023-3, 160951/2022-4,  308395/2023-9, and by Paraíba State Research Foundation (FAPESQ), Grant 3034/2021.

\noindent {\bf Data Availibility:} No data was used for the research described in the article.

\noindent {\bf Declarations}

\noindent {\bf Conflict of interest:} The authors declare no conflict of interest.

\noindent {\bf Declaration of generative AI and AI-assisted technologies in the writing process:} During the preparation of this work, the authors used the AI ChatGPT and the AI Bard to correct grammar and orthography in the text. After using this tool/service, the authors reviewed and edited the content as needed and takes full responsibility for the content of the publication.


\begin{thebibliography}{30}



\bibitem{alvinomercaldo2010neumanndatadoamin} A. Alvino, A. Cianchi, V.G. Maz'ya, and A. Mercaldo, Well-posed elliptic Neumann problems involving irregular data and domains, {\it Ann. Inst. H. Poincaré Anal. Non Linéaire}, {\bf 27}(2010), 1017-1054.


\bibitem{alvinomercaldo2008nonlinearsymetriza} A. Alvino and A. Mercaldo, Nonlinear elliptic problems with $L^1$ data: an approach via symmetrization methods, {\it Mediterr. J. Math.}, {\bf 5}(2008), 173-185.





 \bibitem{andreu1997quasiellipandparab} F. Andreu, J.M. Mazón, S. Segura de León, and J. Toledo, Quasi-linear elliptic and parabolic equations in $L^1$ with nonlinear boundary conditions, {\it Adv. Math. Sci. Appl.}, {\bf 7}(1997), 183–213.
 
 \bibitem{andreuigbida2007existandunqi} F. Andreu, N. Igbida, J. M. Mazón, and J. Toledo,   \(L^{1}\) existence and uniqueness results for quasi-linear elliptic equations with nonlinear boundary conditions, {\it Ann. I. H. Poincaré}, {\bf 24}(2007), 61-89.
 
 


\bibitem{azroulbarbara2013renormalizedp(x)boundrcond} E. Azroul, A. Barbara, M. B. Benboubker, and S. Ouaro, Renormalized solutions for a $p(x)$-Laplacian equation with Neumann nonhomogeneous boundary conditions and $L^1$-data, {\it An. Univ. Craiova Ser. Mat. Inform.}, {\bf 40}(2013), 9-22.
 

\bibitem{benilan1995L1theory} P. Bénilan, L. Boccardo, Th. Gallouët, R. Gariepy, M. Pierre, J.L. Vázquez, An $L^1$-theory of existence and uniqueness of solutions of nonlinear elliptic equations, {\it Ann. Scuola Norm. Sup. Pisa Cl. Sci.}, {\bf 22}(1995), 241–273.


\bibitem{benali2006noncoerciveintegra} M. Ben, Cheikh Ali, and O. Guibé, Nonlinear and non-coercive elliptic problems with integrable data, {\it Adv. Math. Sci. Appl.}, {\bf 16}(2006), 275-297.



\bibitem{bettamercalmura2002measuredatum} M.F. Betta, A. Mercaldo, F. Murat, and  M.M. Porzio, Existence and uniqueness results for nonlinear elliptic problems with a lower order term and measure datum, {\it C. R. Math. Acad. Sci. Paris}, {\bf 334}(2002), 757-762.




\bibitem{betta2002existenrenor} M.F. Betta, A. Mercaldo, F.  Murat, and M.M. Porzio,  Existence of renormalized solutions to nonlinear elliptic equations with lower-order terms and right-hand side measure, {\it J. Math. Pures Appl.}, {\bf 81}(2002), 533-566.

\bibitem{betta2015neumannprob} M.F. Betta, O. Guibé and A. Mercaldo, Neumann problems for nonlinear elliptic equations with $L^1$ data. {\it J. Differential Equations}, {\bf 259}(2015), 898-924.



\bibitem{boccardogall1989nonlinearparab} L. Boccardo and T. Gallouët, Nonlinear elliptic and parabolic equations involving measure data, {\it J. Funct. Anal.}, {\bf 87}(1989), 149–169.


\bibitem{boccardo1992nonlinearhandsidemeasure} I. Boccardo and T.  Gallouët,   Nonlinear elliptic equations with right hand side measures, {\it Comm. Partial Differential Equations}, {\bf 17}(1992), 641-655.



\bibitem{boccardogall1996existentrpy} L. Boccardo, T. Gallouët, and L. Orsina, Existence and uniqueness of entropy solutions for nonlinear elliptic equations with measure data, {\it Ann. Inst. Henri Poincaré Anal. Non Linéaire}, {\bf 13}(1996), 539–551.

\bibitem{castillo2021classicalmultidi} R.E. Castillo and H.C. Chaparro, Classical and Multidimensional Lorentz Spaces. {\it De Gruyter}, 2021.



\bibitem{chabrowski2007neumanndata} J. Chabrowski, On the Neumann problem with $L^1$ data, {\it Colloq. Math.}, {\bf 107}(2007), 301-316.




\bibitem{dallaglio1996approximatesoutil1} A. Dall’Aglio, Approximated solutions of equations with $L^1$ data. Application to the $H$-convergence of quasi-linear
parabolic equations, {\it Ann. Mat. Pura Appl.},  {\bf 170}(1996), 207–240.


\bibitem{dalmasoorsinda1999renorm} G. dal Maso, F. Murat, L. Orsina, and A. Prignet, Renormalized solutions for elliptic equations with general measure data, {\it Ann. Sc. Norm. Super Pisa Cl. Sci.}, {\it 28}(1999), 741–808.


\bibitem{decarreau1996traceimbeddings} A. Decarreau, J. Liang, and J.-M. Rakotoson, Trace imbeddings for $T$-sets and application to Neumann-Dirichlet problems with measures included in the boundary data, {\it Ann. Fac. Sci. Toulouse Math.}, {\bf 5}(1996), 443-470.


\bibitem{vecchio1995nonlinearmeasuredata} T. Del Vecchio, Nonlinear elliptic equations with measure data, {\it Potential Analysis},  {\bf 4}(1995) 185-203. 




\bibitem{droniou2000solvingdualitymethid} J. Droniou, Solving convection-diffusion equations with mixed, Neumann and Fourier boundary conditions and measures as data, by a duality method, {\it Adv. Differential Equations}, {\it 5}(2000), 1341-1396.


\bibitem{droniou2009noncoerciveneuamnnboundary} J. Droniou and  J.-L. Vázquez, Noncoercive convection-diffusion elliptic problems with Neumann boundary conditions, {\it Calc. Var. Partial Differential Equations}, {\bf 34}(2009), 413-434.

\bibitem{edmundsopic1993poincare} D.E. Edmunds and  B. Opic,  Weighted Poincaré and Friedrichs inequalities, {\it J. Lond. Math. Soc.}, {\bf 2}(1993), 79-96.



\bibitem{feronemercaldo1998formulaintegrals} V. Ferone and A. Mercaldo, A second order derivation formula for functions defined by integrals, {\it C. R. Acad. Sci. Paris Sér. I Math.}, {\bf 326}(1998), 549-554.


\bibitem{verone2005neumannsymetr} V. Ferone and A. Mercaldo, Neumann problems and Steiner symmetrization, {\it Comm. Partial Differential Equations}, {\bf 30}(2005), 1537-1553.




\bibitem{fukushimasato1991closable} M. Fukushima, K. Sato, and S. Taniguchi, On the closable part of pre-Dirichlet forms and finite support of the underlying measures, {\it Osaka J. Math.}, {\bf 28}(1991) 517–535.

\bibitem{gianchiedmund1996weightedpoincar} A. Gianchi, D.E. Edmunds, and P. Gurka,   On weighted Poincaré inequalities, {\it Math. Nachr.}, {\bf 180}(1996), 15-41.

\bibitem{guibemercal2006noncoercive} O. Guibé and A. Mercaldo, Existence and stability results for renormalized solutions to noncoercive nonlinear elliptic equations with measure data, {\it Potential Anal.}, {\bf 25}(2006), 223-258.

\bibitem{guibemercal2008existenceofrenormal} O. Guibé and A. Mercaldo, Existence of renormalized solutions to nonlinear elliptic equations with two lower order terms and measure data, {\it Trans. Amer. Math. Soc.}, {\bf 360}(2008), 643-669.


\bibitem{gurka1991continuous} P. Gurka and B. Opic, Continuous and compact imbeddings of weighted Sobolev spaces. III, {\it Czechoslovak Math. J.}, {\bf 41}(1991), 317-341.


\bibitem{horiuchi1989imbedding} T. Horiuchi, The imbedding theorems for weighted Sobolev spaces, {\it J. Math. Kyoto Univ.}, {\bf 29}(1989), 365-403.



\bibitem{ibrangoouaro2016anisotropic} I. Ibrango and S. Ouaro, Entropy solutions for nonlinear elliptic anisotropic problems with homogeneous Neumann boundary condition, {\it J. Appl. Anal. Comput.}, {\bf 6}(2016), 271-292.

 

\bibitem{leray1965quelquesresulatat} J. Leray and J.-L. Lions, Quelques résulatats de Višik sur les problèmes elliptiques non linéaires par les méthodes de Minty-Browder, {\it Bull. Soc. Math. France}, {\bf 93}(1965), 97-107.



\bibitem{lions1969quelquesmethodes} J.-L. Lions, Quelques méthodes de résolution des problèmes aux limites non linéaires, Dunod, 1969.


\bibitem{lionsmuratmanuscriptsurlessolutions} P.-L. Lions and F. Murat, Sur les solutions renormalisées d’équations elliptiques non linéaires, in manuscript.



\bibitem{lionstoappearrenormaliseesnonlineares} P.-L. Lions and F. Murat, Solutions renormalisées d'équations elliptiques non linéaires, to appear.




\bibitem{liu2008compact} Q. Liu, Compact trace in weighted variable exponent Sobolev spaces $W^{1, p(x)}(\Omega ; v_0, v_1)$, {\it J. Math. Anal. Appl.}, {\bf 348}(2008), 760-774.


\bibitem{murat1993renormalizedednolineal} F. Murat, Soluciones renormalizadas de EDP elipticas no lineales, {\it Preprint}, {\bf 93023}, 1993.

\bibitem{murat1994equatioslinearesavecl1} F. Murat, Équations elliptiques non linéaires avec second membre $L^1$ ou mesure, in: Actes du 26ème Congrès National d'Analyse Numérique, Les Karellis, France, 1994, pp. A12-A24.




\bibitem{nyanquiniouar2012entropyfouri} I. Nyanquini and S. Ouaro, Entropy solution for nonlinear elliptic problem involving variable exponent and Fourier type boundary condition, {\it Afr. Mat.}, {\bf 23}(2012), $205-228$.



\bibitem{pfluger1998compact} K. Pfl\"{u}ger, Compact traces in weighted Sobolev spaces, {\it Analysis}, {\bf 18}(1998), 65-84.


\bibitem{prognet1997condtauxlimhomo} A. Prignet, Conditions aux limites non homogènes pour des problèmes elliptiques avec second membre mesure, {\it Ann. Fac. Sci. Toulouse Math.}, {\bf 6}(1997), 297–318.

\bibitem{stampacchia1965problemecoefdiscont} G. Stampacchia, Le problème de Dirichlet pour les équations elliptiques du second ordre à coefficients discontinus, {\it Ann. Inst. Fourier},  {\bf 15}(1965) 189–258.



\bibitem{wittboldzimmer2010exisrenorvariablel1} P. Wittbold and A. Zimmermann, Existence and uniqueness of renormalized solutions to nonlinear elliptic equations with variable exponent and $L^1$-data, {\it Nonlinear Anal. Theory Methods Appl.}, {\bf 72}(2010), 2990-3008.









\end{thebibliography}
\end{document}